\documentclass{amsart}
\usepackage{amsthm,amsmath,amssymb,amscd}

\usepackage{tikz}
\usetikzlibrary{matrix,arrows,calc,snakes,patterns,decorations.markings}

\usepackage{enumitem}
\usepackage[labelsep=space]{caption}

\usepackage{mathtools}
\usepackage{thmtools}
\usepackage{stmaryrd}
\usepackage{mathrsfs}
\usepackage{eucal}
\usepackage{hyperref}
\numberwithin{equation}{section}

\newcommand{\xrightarrowdbl}[2][]{%
  \xrightarrow[#1]{#2}\mathrel{\mkern-14mu}\rightarrow
}

\calclayout

\usepackage[printwatermark]{xwatermark}
\usepackage{xcolor}
\usepackage{graphicx}

\newcommand{\ortriangle}[2]{
\begin{tikzpicture}[scale=#1,baseline=#2]
\draw[fill,black] (0.5,1) circle (0.5mm);
\draw[fill,black] (-0.1,0) circle (0.5mm);
\draw[fill,black] (1.1,0) circle (0.5mm);
\draw[shorten <= 1mm, shorten >= 1mm, ->] (0.5,1) -- (-0.1,0);
\draw[shorten <= 1mm, shorten >= 1mm, ->] (0.5,1) -- (1.1,0);
\draw[shorten <= 1mm, shorten >= 1mm, ->] (-0.1,0) -- (1.1,0);
\end{tikzpicture}
}

\newcommand{\veps}{\ensuremath{\varepsilon}}

\newcommand{\orsquare}[2]{
\begin{tikzpicture}[scale=#1,baseline=#2]
\draw[fill,black] (0,1) circle (0.5mm);
\draw[fill,black] (1,1) circle (0.5mm);
\draw[fill,black] (0,0) circle (0.5mm);
\draw[fill,black] (1,0) circle (0.5mm);
\draw[shorten <= 1mm, shorten >= 1mm, ->] (0,1) -- (0,0);
\draw[shorten <= 1mm, shorten >= 1mm, ->] (0,0) -- (1,0);
\draw[shorten <= 1mm, shorten >= 1mm, ->] (1,1) -- (1,0);
\draw[shorten <= 1mm, shorten >= 1mm, ->] (0,1) -- (1,1);
\draw[shorten <= 1mm, shorten >= 1mm, ->] (0,1) -- (1,0);
\end{tikzpicture}
}

\tikzset{Lshift/.style={to path={%
    ($($(\tikztostart)!0.9mm!+90:(\tikztotarget)$)!0.2!($(\tikztotarget)!0.9mm!-90:(\tikztostart)$)$) --
    ($($(\tikztostart)!0.9mm!+90:(\tikztotarget)$)!0.8!($(\tikztotarget)!0.9mm!-90:(\tikztostart)$)$) \tikztonodes
}}}
\tikzset{Rshift/.style={to path={%
    ($($(\tikztostart)!-0.9mm!+90:(\tikztotarget)$)!0.2!($(\tikztotarget)!-0.9mm!-90:(\tikztostart)$)$) --
    ($($(\tikztostart)!-0.9mm!+90:(\tikztotarget)$)!0.8!($(\tikztotarget)!-0.9mm!-90:(\tikztostart)$)$) \tikztonodes
}}}
\tikzset{Lshifttight/.style={to path={%
    ($($(\tikztostart)!0.9mm!+90:(\tikztotarget)$)!0.1!($(\tikztotarget)!0.9mm!-90:(\tikztostart)$)$) --
    ($($(\tikztostart)!0.9mm!+90:(\tikztotarget)$)!0.9!($(\tikztotarget)!0.9mm!-90:(\tikztostart)$)$) \tikztonodes
}}}
\tikzset{Rshifttight/.style={to path={%
    ($($(\tikztostart)!-0.9mm!+90:(\tikztotarget)$)!0.1!($(\tikztotarget)!-0.9mm!-90:(\tikztostart)$)$) --
    ($($(\tikztostart)!-0.9mm!+90:(\tikztotarget)$)!0.9!($(\tikztotarget)!-0.9mm!-90:(\tikztostart)$)$) \tikztonodes
}}}

\declaretheorem[style=plain,numberwithin=section]{proposition}
\declaretheorem[style=plain,sibling=proposition]{definition}
\declaretheorem[style=plain,numbered=no,name=Definition]{definition*}
\declaretheorem[style=definition,qed=$\lozenge$,sibling=proposition]{example}
\declaretheorem[style=definition,qed=$\lozenge$,sibling=proposition]{notation}

\declaretheorem[style=plain,sibling=proposition]{lemma}
\declaretheorem[style=plain,numbered=no,name=Proposition]{proposition*}
\declaretheorem[style=plain,sibling=proposition]{theorem}
\declaretheorem[style=plain,name=Theorem]{maintheorem}
\declaretheorem[style=plain,numbered=no,name=Theorem]{theorem*}

\declaretheorem[style=plain,numbered=no,name=Resultat]{resultat*}
\declaretheorem[style=plain,numbered=no,name=Result]{result*}
\declaretheorem[style=plain,numbered=no,name=Observation]{observation*}
\declaretheorem[style=plain,numbered=no,name=Question]{question*}
\declaretheorem[style=plain,numbered=no,name=Acknowledgements]{acknowledgements*}
\declaretheorem[style=plain,numbered=no,name=Beobachtung]{beobachtung*}

\declaretheorem[style=plain,sibling=proposition]{corollary}
\declaretheorem[style=plain,sibling=proposition]{fact}

\newcommand{\locarrow}[1]{\xrightarrow{#1}}

\DeclareMathOperator{\Loc}{Loc}

\DeclareMathOperator{\comp}{real}
\DeclareMathOperator{\flatmod}{flat}
\DeclareMathOperator{\lift}{H}

\DeclareMathOperator{\dg}{dg-}
\DeclareMathOperator{\dw}{dw-}
\DeclareMathOperator{\cosyz}{Q}
\DeclareMathOperator{\syz}{Z}
\DeclareMathOperator{\bfR}{R}
\DeclareMathOperator{\bfL}{L}
\DeclareMathOperator{\bfD}{D}
\DeclareMathOperator{\bfK}{K}
\DeclareMathOperator{\Ch}{Ch}

\DeclareMathOperator{\Ho}{Ho}

\DeclareMathOperator{\adjL}{L}
\DeclareMathOperator{\adjR}{R}

\DeclareMathOperator{\ac}{ac}

\DeclareMathOperator{\Acyc}{Acyc}

\DeclareMathOperator{\ProjInj}{Proj-Inj}

\DeclareMathOperator{\co}{co}
\DeclareMathOperator{\sing}{sing}
\DeclareMathOperator{\ctr}{ctr}
\DeclareMathOperator{\opproj}{proj}
\DeclareMathOperator{\gproj}{G-Proj}
\DeclareMathOperator{\ginj}{G-inj}
\DeclareMathOperator{\opinj}{inj}

\DeclareMathOperator{\Inj}{Inj}

\DeclareMathOperator{\id}{id}

\DeclareMathOperator{\Proj}{Proj}

\DeclareMathOperator{\Mod}{-Mod}

\DeclareMathOperator{\Hom}{Hom}

\DeclareMathOperator{\coker}{coker}
\DeclareMathOperator{\image}{im}

\DeclareMathOperator{\filt}{filt-}
\DeclareMathOperator{\GProj}{G-Proj}

\DeclareMathOperator{\ext}{Ext}

\newcommand   {\lac} {\ensuremath{_{\ac}}}

\newcommand   {\uco} {\ensuremath{^{\co}}}
\newcommand {\lsing} {\ensuremath{_{\sing}}}
\newcommand {\uGProj}{\ensuremath{\underline{\GProj}}}

\newcommand  {\uctr} {\ensuremath{^{\ctr}}}

\newcommand {\uproj} {\ensuremath{^{\opproj}}}
\newcommand{\ugproj} {\ensuremath{^{\gproj}}}
\newcommand {\uginj} {\ensuremath{^{\ginj}}}

\newcommand  {\uinj} {\ensuremath{^{\opinj}}}
\newcommand   {\inj} {\ensuremath{\calI}}
\newcommand  {\proj} {\ensuremath{\calP}}

\newcommand {\wt}[1]{\ensuremath{\widetilde{#1}}}

\newcommand    {\ZZ} {\ensuremath{{\mathbb Z}}}
\newcommand  {\calO} {\ensuremath{{\mathcal O}}}

\newcommand  {\calM} {\ensuremath{{\mathcal M}}}
\newcommand  {\calP} {\ensuremath{{\mathcal P}}}

\newcommand  {\calI} {\ensuremath{{\mathcal I}}}

\newcommand  {\calC} {\ensuremath{{\mathcal C}}}
\newcommand  {\calD} {\ensuremath{{\mathcal D}}}
\newcommand  {\calQ} {\ensuremath{{\mathcal C}}}
\newcommand  {\calR} {\ensuremath{{\mathcal D}}}
\newcommand  {\calE} {\ensuremath{{\mathcal E}}}
\newcommand  {\calF} {\ensuremath{{\mathcal F}}}

\newcommand  {\calS} {\ensuremath{{\mathcal S}}}

\newcommand  {\calT} {\ensuremath{{\mathcal T}}}
\newcommand  {\calW} {\ensuremath{{\mathcal W}}}

\newcommand {\scA} {\ensuremath{{\mathscr A}}}
\newcommand {\scF} {\ensuremath{{\mathscr F}}}

\newcommand   {\p}    {\ensuremath{^{\prime}}}
\newcommand  {\pp}    {\ensuremath{^{\prime\prime}}}
\newcommand  {\ppp}    {\ensuremath{^{\prime\prime\prime}}}
\newcommand  {\ua}    {\ensuremath{^{\ast}}}

\newcommand  {\us}    {\ensuremath{^{\sharp}}}

\newcommand  {\ul} [1]{\ensuremath{\underline{#1}}}

\newcommand{\rec}{\ensuremath{
    \begin{tikzpicture}[baseline]
      \useasboundingbox (0,0) rectangle (4.75mm,2mm);
      \draw[->](0.5mm,1.1mm)--(4.25mm,1.1mm);
      \draw[->](4.25mm,0.2mm)--(0.5mm,0.2mm);
      \draw[->](4.25mm,2mm)--(0.5mm,2mm);
    \end{tikzpicture}
  }
}
\newcommand{\bigrec}{\ensuremath{
    \begin{tikzpicture}[baseline]
      \useasboundingbox (0,0) rectangle (9.75mm,2.3mm);
      \draw[->](0.5mm,1.2mm)--(9.25mm,1.2mm);
      \draw[->](9.25mm,0.1mm)--(0.5mm,0.1mm);
      \draw[->](9.25mm,2.3mm)--(0.5mm,2.3mm);
    \end{tikzpicture}
  }
}

\title{\mbox{A realization functor for abelian model categories}}
\author{Hanno Becker}
\email{habecker@math.uni-bonn.de}
\date{\today}

\begin{document}

\begin{abstract}
We study liftings of abelian model structures to categories of chain complexes and construct a realization functor
$\bfD(\scA)\to\Ho(\calM)$ for any cofibrantly generated, hereditary abelian model structure $\calM$ on a Grothendieck
category $\scA$.
\end{abstract}

\maketitle

\section{Introduction}

Given a Grothendieck abelian category $\scA$, we may form its derived category $\bfD(\scA)$, which -- very vaguely --
one might think of as underlying a generic homotopy theory built upon $\scA$. While we don't know how to elaborate this
into a precise statement, this article deals with the following approximation:

\begin{question*}[Realization Problem] If $\calM$ is a 'reasonable' hereditary abelian
model structure on $\scA$ with triangulated homotopy category $\Ho(\calM)$, does there exist a triangulated functor
$\bfD(\scA)\to\Ho(\calM)$ such that the following diagram commutes up to isomorphism?
\begin{equation*}\begin{tikzpicture}[baseline,description/.style={fill=white,inner sep=2pt}]
    \matrix (m) [matrix of math nodes, row sep=1.5em,
                 column sep=2.5em, text height=1.5ex, text depth=0.25ex,
                 inner sep=0pt, nodes={inner xsep=0.3333em, inner ysep=0.3333em}]
    {
       \Ho(\calM) && \bfD(\scA)\\
       & \scA & \\
    };
    \draw[->,rounded corners=4mm] (m-2-2) -| (m-1-3);
    \draw[->,rounded corners=4mm] (m-2-2) -| (m-1-1);
    \draw[->] (m-1-3) -- node[above,scale=0.75]{$\comp$} (m-1-1);
\end{tikzpicture}\end{equation*}
\end{question*}

This question will be answered affirmatively for cofibrantly generated, hereditary abelian model structures on
Grothendieck categories in this article. \vskip1mm

\subsection*{Outline} In this introduction, we will begin by recollecting some by now classical results of Gorenstein homological algebra
leading to a solution of the realization problem for the Gorenstein-projective model structure over an Iwanaga-Gorenstein
ring. The rest of the introduction will then indicate how these results can be generalized, and the details and proofs
of our three main theorems will be supplemented in the remaining sections.

\subsection*{Acknowledgements}

The results of this article are mostly part of my PhD thesis written under the supervision of Prof. Catharina Stroppel at the
University of Bonn. I am very grateful to Catharina for her help and support in all phases of the
project. I am also grateful to Jan Stovicek for a very stimulating discussion at the beginning of my work on the topic
of this article, and to Pieter Belmans for numerous helpful comments.

\subsection{Gorenstein homological algebra via triangulated categories}

Let $R$ be an Iwanaga-Gorenstein ring \cite[Definition 9.1.1]{EnochsJenda}, i.e. $R$ is Noetherian and of finite injective
dimension as a left and as a right module over itself. Further, denote $$\GProj(R) := \{X\in R\Mod\ |\ \forall\ i>0:\
\ext^i_R(X,R)=0\}$$  its Frobenius category of \textit{Gorenstein-projective} $R$-modules, and $$\uGProj(R) := \GProj(R)/\Proj(R)$$ its
\textit{stable category}, a triangulated category. We have the following well-known theorem:
\begin{theorem}\label{theorem_basicA}
There is an equivalence of triangulated categories $$\lift: \uGProj(R)\ \cong\ \bfK\lac(\Proj(R)),$$ where the rhs.~is
the homotopy category of acyclic complexes of projective $R$-modules.
\end{theorem}
This is a consequence of $\GProj(R)$ being a Frobenius category with projective-injective objects given by $\Proj(R)$
together with Theorem \ref{theorem_frobeniusclassical} below.\vskip2mm

Next, the chain-complex description of the stable category relates to the ordinary derived category through a recollement:
\begin{theorem}[\protect{\cite[Theorem 5.15]{Murfet_MockHomotopyCategoryOfProjectives}, \cite[Corollary
    4.3]{Krause_StableDerived}}]\label{theorem_basicB}
There are recollements of triangulated categories
\begin{align*}\bfK\lac(\Proj(R))\bigrec\bfK(\Proj(R))\bigrec\bfD(R)\\
  \bfK\lac(\Inj(R))\bigrec\bfK(\Inj(R))\bigrec\bfD(R)\end{align*}
\end{theorem}
In fact, the injective recollement exists in a considerably more general situation, see
\cite{Krause_StableDerived}. Finally, we have the following description of the stabilization functor \cite[\S
5]{Krause_StableDerived} of the recollement from Theorem
\ref{theorem_basicB}. In the remainder of this paper, left-directed left resp. right adjoints in
recollements we be denoted $\lambda$ resp. $\rho$.

\begin{theorem}[\protect{\cite[\S 7]{Krause_StableDerived}}]\label{theorem_basicC}
The composition $$R\Mod\hookrightarrow\bfD(R)\xhookrightarrow{\ \ \rho\ \ }\bfK(\Proj(R)) \xrightarrowdbl{\ \ \lambda\ \ }\bfK\lac(\Proj(R))
\xrightarrow{\ \ \lift\ \ }\uGProj(R)$$ is equivalent to the \textit{Gorenstein-projective approximation}
\end{theorem}

Here, by Gorenstein-projective approximation the right adjoint $\ul{R\Mod}\to\uGProj(R)$ to the inclusion functor is meant.\vskip2mm

As the next section will show, Theorem \ref{theorem_basicC} is essentially the solution to a realization problem in
the present situation. Also, note that Theorems \ref{theorem_basicA}, \ref{theorem_basicB}, \ref{theorem_basicC} have
analogues in the Gorenstein-injective case.

\subsection{Gorenstein homological algebra via model categories}

The results of the previous section can be lifted to the world of abelian model categories and thereby be seen as intimately
related to a suitable realization problem in this context.

For brevity, we refer to
\cite{Becker_ModelsForSingularityCategories} for the basics on abelian model categories and also keep the notation of
op.~cit.; in particular, abelian model structures will be denoted $\calM$, and their classes of cofibrant, weakly trivial
and fibrant objects in an abelian model category will be denoted $\calC$, $\calW$ and $\calF$, respectively.

The first step is the introduction of an abelian model structure for $\uGProj(R)$:
\begin{theorem}[\protect{\cite[Theorem 8.6]{Hovey_Cotorsion}}]
Let $R$ be a Gorenstein ring.
\begin{enumerate}
\item There exists an abelian model structure on $R\Mod$, called the \textit{Gorenstein projective} model
  structure and denoted $\calM\ugproj(R)$, with $\calC = \gproj(R)$, $\calW = \proj^{<\infty}(R)$ (the modules of finite
  projective dimension) and $\calF=R\Mod$.
\item There exists an abelian model structure on $R\Mod$, called the \textit{Gorenstein injective} model
  structure and denoted $\calM\uginj(R)$, with $\calC=R\Mod$, $\calW = \proj^{<\infty}(R)$ and $\calF = \ginj(R)$.
\end{enumerate}
Moreover, $\calM\ugproj(R)$ and $\calM\uginj(R)$ are cofibrantly generated, and their homotopy categories are
canonically equivalent to the stable categories of Gorenstein-projective resp. Gorenstein-injective $R$-modules.
\end{theorem}

The previous section solves the realization problem for $\calM\ugproj(R)$: The
Gorenstein-projective approximation of an $R$-module is its cofibrant replacement in the Gorenstein-projective model
structure $\calM\ugproj(R)$, so Theorem \ref{theorem_basicC} shows that the diagram
\begin{equation*}\begin{tikzpicture}[baseline,description/.style={fill=white,inner sep=2pt}]
    \matrix (m) [matrix of math nodes, row sep=1.5em,
                 column sep=2.5em, text height=1.5ex, text depth=0.25ex,
                 inner sep=0pt, nodes={inner xsep=0.3333em, inner ysep=0.3333em}]
    {
       \Ho(\calM\ugproj(R))\cong \uGProj(R) && \bfD(R\Mod)\\
       & R\Mod & \\
    };
    \draw[->,rounded corners=4mm] (m-2-2) -| (m-1-3);
    \draw[->,rounded corners=4mm] (m-2-2) -| ($(m-1-1) - (1cm,3mm)$);
    \draw[->] (m-1-3) -- node[above,scale=0.75]{$\lift\circ\lambda\circ\rho$} (m-1-1);
\end{tikzpicture}\end{equation*}
commutes up to isomorphism, as wished.

The following theorem lifts Theorem \ref{theorem_basicB} to the level of abelian model categories. Its generalizability
to suitable abelian model structures will be the key to constructing a realization functor in the general setting.

\begin{theorem}[\protect{\cite[Proposition 2.2.4]{Becker_ModelsForSingularityCategories}}]
The injective recollement associated with a Gorenstein ring $R$ admits a lifting to a \emph{butterfly} of
abelian model structures (see below):
\begin{align*}
\begin{tikzpicture}[baseline,description/.style={fill=white,inner sep=2pt}]
    \matrix (m) [matrix of math nodes, row sep=1.8em,
                 column sep=4em, text height=1.5ex, text depth=0.25ex,
                 inner sep=0pt, nodes={inner xsep=0.3333em, inner ysep=0.3333em}]
    {
       \calM\uco\lsing(R) \pgfmatrixnextcell\pgfmatrixnextcell \calM\uinj(R)\\
       \pgfmatrixnextcell \calM\uco(R) \pgfmatrixnextcell \\
       {^{i}}\calM\uco\lsing(R) \pgfmatrixnextcell\pgfmatrixnextcell {^{m}}\calM\uinj(R)\\
    };
    \draw[->] ($(m-1-1.south) + (1mm,0)$) -- node[scale=0.75,right]{L} ($(m-3-1.north) + (1mm,0)$);
    \draw[->] ($(m-3-1.north) - (+1mm,0)$) -- node[scale=0.75,left]{R} ($(m-1-1.south) - (1mm,0)$);
    \draw[->] ($(m-1-3.south) + (1mm,0)$) -- node[scale=0.75,right]{R} ($(m-3-3.north) + (1mm,0)$);
    \draw[->] ($(m-3-3.north) + (-1mm,0)$) -- node[scale=0.75,left]{L} ($(m-1-3.south) - (1mm,0)$);
    \draw[->] ($(m-1-1.south east) - (4mm,0)$) -- node[scale=0.75,below]{L} ($(m-2-2.north west) + (0mm,0)$);
    \draw[->] ($(m-2-2.north west) + (4mm,0)$) -- node[scale=0.75,above]{R} ($(m-1-1.south east) - (0mm,0)$);
    \draw[->] ($(m-3-1.north east) + (0mm,0)$) -- node[scale=0.75,below]{R} ($(m-2-2.south west) + (4mm,0)$);
    \draw[->] ($(m-2-2.south west) + (0mm,0)$) -- node[scale=0.75,above]{L} ($(m-3-1.north east) - (4mm,0)$);
    \draw[->] ($(m-2-2.north east) + (0mm,0)$) -- node[scale=0.75,below]{L} ($(m-1-3.south west) + (4mm,0)$);
    \draw[->] ($(m-1-3.south west) + (0mm,0)$) -- node[scale=0.75,above]{R} ($(m-2-2.north east) - (4mm,0)$);
    \draw[->] ($(m-2-2.south east) - (4mm,0)$) -- node[scale=0.75,below]{R} ($(m-3-3.north west) + (0mm,0)$);
    \draw[->] ($(m-3-3.north west) + (4mm,0)$) -- node[scale=0.75,above]{L} ($(m-2-2.south east) + (0mm,0)$);
\end{tikzpicture}
\end{align*}
Similarly, the projective recollement from Theorem \ref{theorem_basicB} arises from a butterfly.
\end{theorem}

For the definition of the model structures involved in the previous theorem, see loc.cit.

\begin{definition}\label{def_butterfly}
A diagram of abelian model structures and Quillen adjunctions of shape
\begin{align}\label{eq:butterflygen}\tag{$\bowtie$}
\begin{tikzpicture}[baseline,description/.style={fill=white,inner sep=2pt}]
    \matrix (m) [matrix of math nodes, row sep=1.8em,
                 column sep=4em, text height=1.5ex, text depth=0.25ex,
                 inner sep=0pt, nodes={inner xsep=0.3333em, inner ysep=0.3333em}]
    {
       \ \ \calM_l\ \  \pgfmatrixnextcell\pgfmatrixnextcell \ \ \calM_r\ \ \\
       \pgfmatrixnextcell \ \ \calM\ \  \pgfmatrixnextcell \\
       \ \ \calM_l\p\ \  \pgfmatrixnextcell\pgfmatrixnextcell \ \ \calM_r\p\ \ \\
    };
    \draw[->] ($(m-1-1.south) + (1mm,0)$) -- node[scale=0.75,right]{L} ($(m-3-1.north) + (1mm,0)$);
    \draw[->] ($(m-3-1.north) - (+1mm,0)$) -- node[scale=0.75,left]{R} ($(m-1-1.south) - (1mm,0)$);
    \draw[->] ($(m-1-3.south) + (1mm,0)$) -- node[scale=0.75,right]{R} ($(m-3-3.north) + (1mm,0)$);
    \draw[->] ($(m-3-3.north) + (-1mm,0)$) -- node[scale=0.75,left]{L} ($(m-1-3.south) - (1mm,0)$);
    \draw[->,shorten <=0.25cm] ($(m-1-1.south east) - (4mm,0)$) -- node[scale=0.75,below]{L} ($(m-2-2.north west) + (0mm,0)$);
    \draw[->,shorten <=0.25cm] ($(m-2-2.north west) + (4mm,0)$) -- node[scale=0.75,above]{R} ($(m-1-1.south east) - (0mm,0)$);
    \draw[->,shorten <=0.25cm] ($(m-3-1.north east) + (0mm,0)$) -- node[scale=0.75,below]{R} ($(m-2-2.south west) + (4mm,0)$);
    \draw[->,shorten >=0.25cm] ($(m-2-2.south west) + (0mm,0)$) -- node[scale=0.75,above]{L} ($(m-3-1.north east) - (4mm,0)$);
    \draw[->,shorten >=0.25cm] ($(m-2-2.north east) + (0mm,0)$) -- node[scale=0.75,below]{L} ($(m-1-3.south west) + (4mm,0)$);
    \draw[->,shorten >=0.25cm] ($(m-1-3.south west) + (0mm,0)$) -- node[scale=0.75,above]{R} ($(m-2-2.north east) - (4mm,0)$);
    \draw[->,shorten <=0.25cm] ($(m-2-2.south east) - (4mm,0)$) -- node[scale=0.75,below]{R} ($(m-3-3.north west) + (0mm,0)$);
    \draw[->,shorten <=0.25cm] ($(m-3-3.north west) + (4mm,0)$) -- node[scale=0.75,above]{L} ($(m-2-2.south east) + (0mm,0)$);
\end{tikzpicture}
\end{align}
is called a \emph{butterfly} if the following properties hold:
\begin{enumerate}[leftmargin=1.25cm,label=($\bowtie$.\roman*)]
\item\label{item:butterflyquillgen} The left and right vertical adjunctions are Quillen equivalences.
\item\label{item:butterflywingsgen} The two wings in the following following diagram commute:
\begin{equation}\label{eq:butterflyexactgen}\tag{$\Ho(\bowtie)$}
\begin{tikzpicture}[baseline,description/.style={fill=white,inner sep=2pt}]
    \matrix (m) [matrix of math nodes, row sep=1.8em,
                 column sep=4em, text height=1.5ex, text depth=0.25ex,
                 inner sep=0pt, nodes={inner xsep=0.3333em, inner ysep=0.3333em}]
    {
       \Ho(\calM_l) \pgfmatrixnextcell\pgfmatrixnextcell \Ho\left(\calM_r\right)\\
       \pgfmatrixnextcell \Ho(\calM) \pgfmatrixnextcell \\
       \Ho\left(\calM_l\p\right) \pgfmatrixnextcell\pgfmatrixnextcell \Ho\left(\calM\p_r\right)\\
    };
    \draw[->] (m-1-1) -- node[scale=.75,left]{$\cong$} node[scale=0.75,description]{$\bfL\id$} (m-3-1);
    \draw[->] (m-1-3) -- node[scale=.75,right]{$\cong$} node[scale=0.75,description]{$\bfR\id$} (m-3-3);
    \draw[->] (m-1-1) -- node[scale=0.75,description]{$\bfL\id$} (m-2-2);
    \draw[->] (m-2-2) -- node[scale=0.75,description]{$\bfR\id$} (m-3-3);
    \draw[->] (m-3-1) -- node[scale=0.75,description]{$\bfR\id$} (m-2-2);
    \draw[->] (m-2-2) -- node[scale=0.75,description]{$\bfL\id$} (m-1-3);
\end{tikzpicture}\end{equation}
\item\label{item:butterflyrows2gen} The derived horizontal adjunctions in \eqref{eq:butterflygen} realize the homotopy
  categories of the left and right sides as full subcategories of the homotopy category of the middle term, and the
  horizontal functors in \eqref{eq:butterflyexactgen} form an exact sequence.
\end{enumerate}
\end{definition}

\subsection{Generalization} We now describe our generalizations of Theorems
\ref{theorem_basicA}, \ref{theorem_basicB}, \ref{theorem_basicC}.

\subsubsection*{Setup} In the following, let $\scA$ be a Grothendieck abelian category and let $\calM=(\calC,\calW,\calF)$ be a
cofibrantly generated and hereditary (cgh) abelian model structure on $\scA$. In particular, $\calC\cap\calF$ is a
Frobenius category with respect to the short exact sequences inherited from $\scA$, its class of
projective-injective objects equals $\omega := \calC\cap\calW\cap\calF$, and the homotopy category
$\Ho(\calM)$ is canonically equivalent to the stable category $\calC\cap\calF/\calC\cap\calW\cap\calF$.
\vskip1mm
Recall the following basic theorem on Frobenius categories generalizing Theorem \ref{theorem_basicA}:

\begin{theorem}\label{theorem_frobeniusclassical}
Let $\scF$ be a Frobenius category and $\omega :=\ProjInj(\scF)$ its class of projective-injective objects. Then there
 are equivalences of triangulated categories
$$\cosyz^0, \syz^0: \bfK\lac(\omega)\xrightarrow{\quad}\ul{\scF}=\scF/\omega$$
which coincide up to shift, $\cosyz^0 \cong \Sigma\circ \syz^0$.
\end{theorem}

This suggests that a lifting of $\calM$ to a Quillen equivalent model structure on $\Ch(\scA)$ could be obtained by
providing a choice for an abelian model structure $\wt{\calM}$ on $\Ch(\scA)$ such that the class of acyclic complexes
with values in $\omega := \calC\cap\calW\cap\calF$ is equal to the class of bifibrant objects in $\wt{\calM}$. The base
for implementing this idea will be the following Theorem of Gillespie extending the author's previous work:

\begin{theorem}[\protect{\cite[Proposition 1.4.2]{Becker_ModelsForSingularityCategories}, \cite[Theorem 1.1]{Gillespie_GeneralLocalization}}]\label{thm_gillespie}
Let $\scA$ be an abelian category and $(\calQ,\wt{\calR})$ and $(\wt{\calQ},\calR)$ be complete (small), hereditary
cotorsion pairs over $\scA$ with $\wt{\calQ}\subseteq\calQ$ and $\calQ\cap\wt{\calR}=\wt{\calQ}\cap\calR$. Then there
exists a unique (cofibrantly generated) abelian model structure $(\calQ,\calW,\calR)$, and its class $\calW$ of weakly
trivial objects is given by
\begin{align}\label{eq:thmgillespiew}
\calW = \{X\in\scA\ |\ \exists\ 0\to X\to\wt{\calR}\to\wt{\calQ}\to 0\quad\text{and}\quad 0\to\wt{\calR}\to\wt{\calQ}\to
X\to 0\}.
\end{align}
\end{theorem}
The situation of Theorem \ref{thm_gillespie} turns out to be abundant and deserves its own name:
\begin{notation}\label{notation_locpair}
We write $(\calD\p,\calE\p)\locarrow{}(\calD,\calE)$, or $(\calD\p,\calE\p)\locarrow{\alpha}(\calD,\calE)$, as an abbreviation for $(\calD\p,\calE\p)$ and $(\calD,\calE)$ being complete and hereditary cotorsion pairs
having the same core $\calD\p\cap\calE\p=\calD\cap\calE$ and satisfying $\calD\p\subset\calD$. Such a situation will be
called a \emph{localization context}. Its induced model structure from Theorem \ref{thm_gillespie} is denoted
$\Loc(\alpha) := (\calD,?,\calE\p)$ on $\scA$ and called its \emph{localization}.
\end{notation}

Leaving the details for later, the technical heart of this article is the following zoology of localization contexts
induced by a single cgh abelian model structure. For the notation, we refer to Definition \ref{def_classes}.

\begin{maintheorem}\label{thm_bunch}
Let $\calM=(\calC,\calW,\calF)$ be a cofibrantly generated, hereditary abelian model structure on the Grothendieck
category $\scA$ with core $\omega := \calC\cap\calW\cap\calF$. Then Figure \ref{fig_bunch} shows a diagram of
localization contexts on $\Ch(\scA)$ with common core.

An arrow $(\calD\p,\calE\p)\to(\calD,\calE)$ signifies (independent of its style) that $\calD\p\subseteq\calD$,
and $\perp$ indicates that the corresponding entry is the left/right orthogonal of the other entry.
\end{maintheorem}

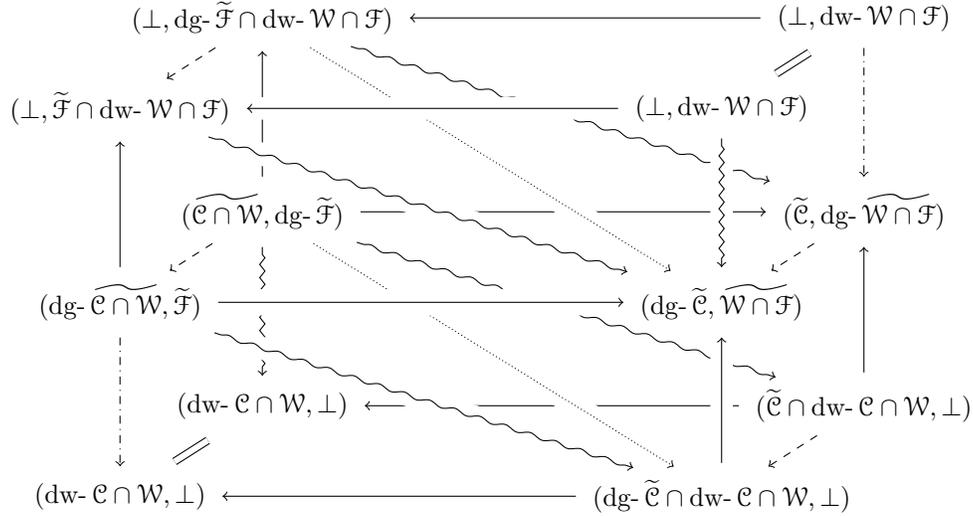
\begin{figure}[h]
\begin{equation*}\begin{tikzpicture}[scale=0.86,baseline,description/.style={fill=white,inner sep=2pt}]
    \begin{scope}[xshift=2.2cm,yshift=-0.1cm,yscale=3]
      \node (m-1-2) at (0,0) {$(\perp,\dg\wt{\calF}\cap\dw\calW\cap\calF)$};
      \node (m-3-2) at (0,-1) {$(\wt{\calC\cap\calW},\dg\wt{\calF})$};
      \node (m-5-2) at (0,-2) {$(\dw\calC\cap\calW,\perp)$};
    \end{scope}
    \begin{scope}[xshift=0cm,yshift=-1.5cm,yscale=3]
      \node (m-2-1) at (0,0) {$(\perp,\wt{\calF}\cap\dw\calW\cap\calF)$};
      \node (m-4-1) at (0,-1) {$(\dg\wt{\calC\cap\calW},\wt{\calF})$};
      \node (m-6-1) at (0,-2) {$(\dw\calC\cap\calW,\perp)$};
    \end{scope}
    \begin{scope}[xshift=11.5cm,yshift=-0.1cm,yscale=3]
      \node (m-1-4) at (0,0) {$(\perp,\dw\calW\cap\calF)$};
      \node (m-3-4) at (0,-1) {$(\wt{\calC},\dg\wt{\calW\cap\calF})$};
      \node (m-5-4) at (0,-2) {$(\wt{\calC}\cap\dw\calC\cap\calW,\perp)$};
    \end{scope}
    \begin{scope}[xshift=9.3cm,yshift=-1.5cm,yscale=3]
      \node (m-2-3) at (0,0) {$(\perp,\dw\calW\cap\calF)$};
      \node (m-4-3) at (0,-1) {$(\dg\wt{\calC},\wt{\calW\cap\calF})$};
      \node (m-6-3) at (0,-2) {$(\dg\wt{\calC}\cap\dw\calC\cap\calW,\perp)$};
    \end{scope}
    \draw[shorten >=0.12cm,shorten <=0.12cm,->,decorate,decoration={snake,amplitude=.25mm, pre length=2mm, post
      length=2mm}] (m-1-2) -- node[near start,above,scale=0.75]{$$} (m-3-4);
    \draw[shorten >=0.12cm,shorten <=0.12cm,->,decorate,decoration={snake,amplitude=.25mm, pre length=2mm, post
      length=2mm}] (m-3-2) -- node[near start,above,scale=0.75]{$$} (m-5-4);
    \draw[shorten >=0.1cm,shorten <=0.1cm,->] (m-1-4) -- node[above,scale=0.75]{$$} (m-1-2);
    \draw[shorten >=0.1cm,shorten <=0.1cm,->] (m-3-2) -- node[near start,above,scale=0.75]{$$} (m-3-4);
    \draw[shorten >=0.1cm,shorten <=0.1cm,->] (m-5-4) -- node[near end, above,scale=0.75]{$$} (m-5-2);
    \draw[shorten >=0.1cm,shorten <=0.1cm,->] (m-3-2) -- node[near end,right,scale=0.75]{$$} (m-1-2);
    \draw[shorten >=0.1cm,shorten <=0.1cm,->,decorate,decoration={zigzag,amplitude=.35mm,segment length=5pt, post
      length=2mm, pre length=2mm}] (m-3-2) -- node[near start,right,scale=0.75]{$$} (m-5-2);
    \draw[shorten >=0.1cm,shorten <=0.1cm,->,dashdotted] (m-1-4) -- node[near start,right,scale=0.75]{$$} (m-3-4);
    \draw[shorten >=0.1cm,shorten <=0.1cm,->] (m-5-4) -- node[near end,right,scale=0.75]{$$} (m-3-4);
    \draw[shorten >=0.1cm,shorten <=0.1cm,line width=3mm,fill,white] (m-1-2) -- (m-4-3);
    \draw[shorten >=0.1cm,shorten <=0.1cm,->,densely dotted] (m-1-2) -- node[above,scale=0.75]{$$} (m-4-3);
    \draw[shorten >=0.1cm,shorten <=0.1cm,line width=3mm,fill,white] (m-3-2) -- (m-6-3);
    \draw[shorten >=0.1cm,shorten <=0.1cm,densely dotted,->] (m-3-2) -- node[above,scale=0.75]{$$} (m-6-3);
    \draw[shorten >=0.1cm,shorten <=0.1cm,->,dashed] (m-1-2) -- node[above=2mm,left,scale=0.75]{$$} (m-2-1);
    \draw[shorten >=0.1cm,shorten <=0.1cm,->,dashed] (m-3-2) -- node[above=2mm,left,scale=0.75]{$$} (m-4-1);
    \draw[shorten >=0.3cm,shorten <=0.3cm,double,double distance=0.8mm] (m-5-2) -- (m-6-1);
    \draw[shorten >=0.1cm,shorten <=0.1cm,->,dashed] (m-5-4) -- node[below=2mm,right,scale=0.75]{$$} (m-6-3);
    \draw[shorten >=0.1cm,shorten <=0.1cm,->,dashed] (m-3-4) -- node[below=2mm,right,scale=0.75]{$$} (m-4-3);
    \draw[shorten >=0.3cm,shorten <=0.3cm,double,double distance=0.8mm] (m-1-4) -- (m-2-3);
    \draw[shorten >=0.1cm,shorten <=0.1cm,->] (m-4-1) -- node[near end,left,scale=0.75]{$$} (m-2-1);
    \draw[shorten >=0.1cm,shorten <=0.1cm,->,dashdotted] (m-4-1) -- node[near start,left,scale=0.75]{$$} (m-6-1);
    \draw[line width=3mm,fill,white] (m-2-3) -- (m-4-3);
    \draw[shorten >=0.1cm,shorten <=0.1cm,->,decorate,decoration={zigzag,amplitude=.35mm,segment length=5pt, post
      length=2mm, pre length=2mm}] (m-2-3) -- node[near start,right,scale=0.75]{$$} (m-4-3);
    \draw[line width=3mm,fill,white] (m-6-3) -- (m-4-3);
    \draw[shorten >=0.1cm,shorten <=0.1cm,->] (m-6-3) -- node[near end,right,scale=0.75]{$$} (m-4-3);
    \draw[line width=3mm,fill,white] (m-6-3) -- (m-6-1);
    \draw[shorten >=0.1cm,shorten <=0.1cm,->] (m-6-3) -- node[near start,above,scale=0.75]{$$} (m-6-1);
    \draw[line width=3.5mm,fill,white] (m-4-1) -- (m-4-3);
    \draw[shorten >=0.1cm,shorten <=0.1cm,->] (m-4-1) -- node[near start,above,scale=0.75]{$$} (m-4-3);
    \draw[line width=3mm,fill,white] (m-2-3) -- (m-2-1);
    \draw[shorten >=0.1cm,shorten <=0.1cm,->] (m-2-3) -- node[near end,above,scale=0.75]{$$} (m-2-1);
    \draw[line width=3mm,fill,white] (m-2-1) -- (m-4-3);
    \draw[shorten >=0.12cm,shorten <=0.12cm,->,decorate,decoration={snake,amplitude=.25mm, pre length=2mm, post
      length=2mm}] (m-2-1) -- node[near start,above,scale=0.75]{$$} (m-4-3);
    \draw[line width=3mm,fill,white] (m-6-3) -- (m-4-1);
    \draw[shorten >=0.12cm,shorten <=0.12cm,->,decorate,decoration={snake,amplitude=.25mm, pre length=2mm, post
      length=2mm}] (m-4-1) -- node[near start,above,scale=0.75]{$$} (m-6-3);
\end{tikzpicture}\end{equation*}
\caption{Localization contexts on $\Ch(\scA)$ induced by a cofibrantly generated, hereditary abelian model structure on $\scA$}\label{fig_bunch}
\end{figure}

With the multitude of localization contexts from Theorem \ref{thm_bunch} and each of them inducing an abelian model
structure on $\Ch(\scA)$ by virtue of Theorem \ref{thm_gillespie}, we end up with a large number of abelian model
structures on $\Ch(\scA)$. The following helps gaining a rough understanding on their interrelation and will be
elaborated later:

\begin{enumerate}[leftmargin=1cm,label=(\roman*)]
\item\label{item:Btriangle} Each triangle $\ortriangle{0.4}{1mm}$ in Figure \ref{fig_bunch} gives rise to a localization sequence between the three
  model structures induced by its edges.
\item\label{item:Bsquare} Each square $\orsquare{0.4}{1mm}$ in Figure \ref{fig_bunch} yields a butterfly-shaped diagram of adjunctions
  between the two localization sequences associated via \ref{item:Btriangle} to its triangle faces. However, these are not necessarily butterflies.
\end{enumerate}

Using this, we get the desired generalization of Theorem \ref{theorem_basicB}:

\begin{maintheorem}\label{thm_chainbutterfly}
For a cofibrantly generated and hereditary abelian model structure $\calM=(\calC,\calW,\calF)$ over a Grothendieck
category $\scA$, consider Figure \ref{fig_chainbutterfly}. It shows a diagram of identity Quillen adjunctions between
cofibrantly generated and hereditary abelian model structures on $\Ch(\scA)$ with the following properties:
\begin{enumerate}
\item Both horizontal layers are butterflies in the sense of Definition \ref{def_butterfly}.
\item All the vertical adjunctions are Quillen equivalences.
\item\label{thm_chainbutterfly_it3} The four model structures on the left side are Quillen equivalent to $\calM$.
\item The four model structures on the right side hand have their homotopy category canonically equivalent
to $\bfD(\scA)$.
\end{enumerate}
\end{maintheorem}

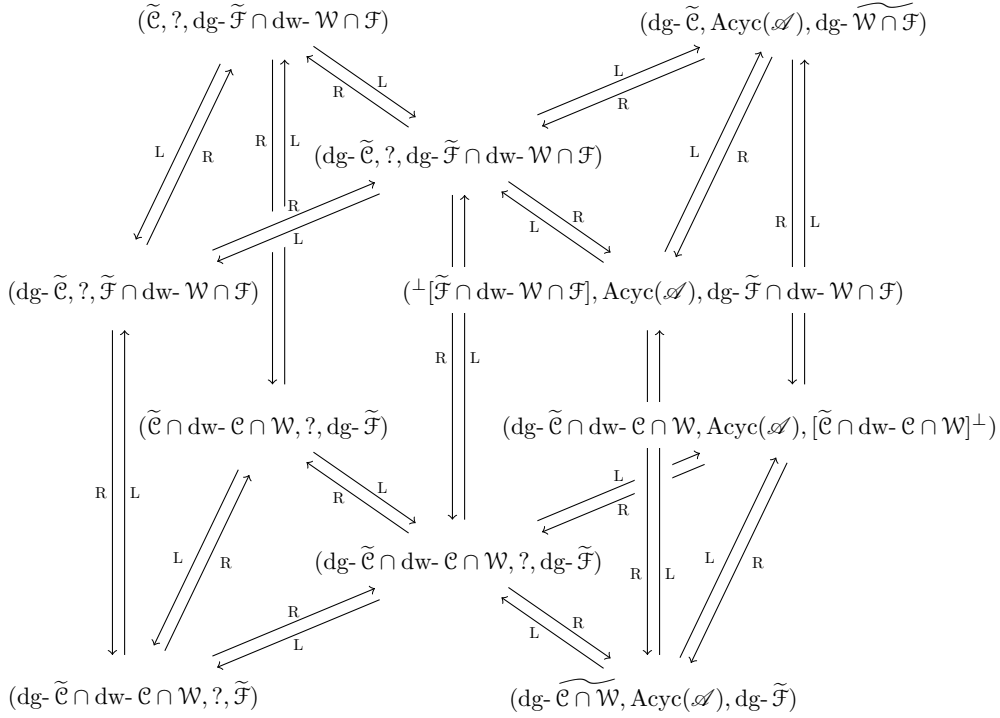
\begin{figure}[h]
\begin{equation*}\begin{tikzpicture}[scale=0.9]
    \begin{scope}[xscale=1.6]
      \node (m-1-3) at (1.2cm,0){};
      \node (m-1-7) at (6cm,0){};
      \node (m-3-4) at (3cm,-2cm){};
      \node (m-5-1) at (0cm,-4cm){};
      \node (m-5-5) at (4.8cm,-4cm){};
    \end{scope}
    \begin{scope}[yshift=-6cm,xscale=1.6]
      \node (m-6-3) at (1.2cm,0){};
      \node (m-6-7) at (6cm,0){};
      \node (m-8-4) at (3cm,-2cm){};
      \node (m-10-1) at (0cm,-4cm){};
      \node (m-10-5) at (4.8cm,-4cm){};
    \end{scope}
    \draw[shorten >=0.7cm,shorten <=0.7cm,<-] (m-3-4) to[Rshifttight] node[below,scale=0.6]{$\adjR$} (m-1-7);
    \draw[shorten >=0.7cm,shorten <=0.7cm,<-] (m-1-7) to[Rshifttight] node[above,scale=0.6]{$\adjL$} (m-3-4);
    \draw[shorten >=0.12cm,shorten <=0.12cm,<-] (m-3-4) to[Rshift] node[above,right=1mm,scale=0.6]{$\adjL$} (m-1-3);
    \draw[shorten >=0.12cm,shorten <=0.12cm,<-] (m-1-3) to[Rshift] node[below,left=1mm,scale=0.6]{$\adjR$} (m-3-4);
    \draw[shorten >=0.3cm,shorten <=0.3cm,->,transform canvas={xshift=-2mm}] (m-1-3) to[Rshifttight] node[left=1mm,scale=0.6]{$\adjL$} (m-5-1);
    \draw[shorten >=0.3cm,shorten <=0.3cm,->,transform canvas={xshift=-2mm}] (m-5-1) to[Rshifttight] node[right=1mm,scale=0.6]{$\adjR$} (m-1-3);
    \draw[shorten >=0.12cm,shorten <=0.12cm,->] (m-3-4) to[Lshift] node[above,right=1mm,scale=0.6]{$\adjR$} (m-5-5);
    \draw[shorten >=0.12cm,shorten <=0.12cm,->] (m-5-5) to[Lshift] node[below,left=1mm,scale=0.6]{$\adjL$} (m-3-4);
    \draw[shorten >=0.12cm,shorten <=0.12cm,->] (m-5-5) to[Lshifttight] node[left=1mm,scale=0.6]{$\adjL$} (m-1-7);
    \draw[shorten >=0.12cm,shorten <=0.12cm,->] (m-1-7) to[Lshifttight] node[right=1mm,scale=0.6]{$\adjR$} (m-5-5);

    \draw[shorten >=0.7cm,shorten <=0.7cm,<-] (m-8-4) to[Rshifttight] node[below,scale=0.6]{$\adjR$} (m-6-7);
    \draw[shorten >=0.7cm,shorten <=0.7cm,<-] (m-6-7) to[Rshifttight] node[above,scale=0.6]{$\adjL$} (m-8-4);
    \draw[shorten >=0.12cm,shorten <=0.12cm,<-] (m-8-4) to[Rshift] node[above,right=1mm,scale=0.6]{$\adjL$} (m-6-3);
    \draw[shorten >=0.12cm,shorten <=0.12cm,<-] (m-6-3) to[Rshift] node[below,left=1mm,scale=0.6]{$\adjR$} (m-8-4);
    \draw[shorten >=0.7cm,shorten <=0.7cm,->] (m-8-4) to[Lshifttight] node[below,scale=0.6]{$\adjL$} (m-10-1);
    \draw[shorten >=0.7cm,shorten <=0.7cm,->] (m-10-1) to[Lshifttight] node[above,scale=0.6]{$\adjR$} (m-8-4);
    \draw[shorten >=0.3cm,shorten <=0.3cm,->,transform canvas={xshift=0.4mm}] (m-6-3) to[Rshifttight] node[left=1mm,scale=0.6]{$\adjL$} (m-10-1);
    \draw[shorten >=0.3cm,shorten <=0.3cm,->,transform canvas={xshift=0.4mm}] (m-10-1) to[Rshifttight] node[right=1mm,scale=0.6]{$\adjR$} (m-6-3);
    \draw[shorten >=0.12cm,shorten <=0.12cm,->] (m-8-4) to[Lshift] node[above,right=1mm,scale=0.6]{$\adjR$} (m-10-5);
    \draw[shorten >=0.12cm,shorten <=0.12cm,->] (m-10-5) to[Lshift] node[below,left=1mm,scale=0.6]{$\adjL$} (m-8-4);
    \draw[shorten >=0.12cm,shorten <=0.12cm,->,transform canvas={xshift=+2mm}] (m-10-5) to[Lshifttight] node[left=1mm,scale=0.6]{$\adjL$} (m-6-7);
    \draw[shorten >=0.12cm,shorten <=0.12cm,->,transform canvas={xshift=+2mm}] (m-6-7) to[Lshifttight] node[right=1mm,scale=0.6]{$\adjR$} (m-10-5);

    \draw[shorten >=0cm,shorten <=0cm,->,transform canvas={xshift=-2mm}] (m-5-1)  to[Rshifttight] node[left,scale=0.6]{$\adjR$} (m-10-1);
    \draw[shorten >=0cm,shorten <=0cm,->,transform canvas={xshift=-2mm}] (m-10-1) to[Rshifttight] node[right,scale=0.6]{$\adjL$} (m-5-1);
    \draw[shorten >=0cm,shorten <=0cm,->,transform canvas={xshift=2mm}] (m-1-7)  to[Rshifttight] node[left,scale=0.6]{$\adjR$} (m-6-7);
    \draw[shorten >=0cm,shorten <=0cm,->,transform canvas={xshift=2mm}] (m-6-7)  to[Rshifttight] node[right,scale=0.6]{$\adjL$} (m-1-7);
    \draw[shorten >=0cm,shorten <=0cm,->,transform canvas={xshift=0mm}] (m-3-4)  to[Rshifttight] node[left,scale=0.6]{$\adjR$} (m-8-4);
    \draw[shorten >=0cm,shorten <=0cm,->,transform canvas={xshift=0mm}] (m-8-4)  to[Rshifttight] node[right,scale=0.6]{$\adjL$} (m-3-4);
    \draw[shorten >=0cm,shorten <=0cm,->,transform canvas={xshift=+2mm}] (m-1-3)  to[Rshifttight] node[pos=0.25,left,scale=0.6]{$\adjR$} (m-6-3);
    \draw[shorten >=0cm,shorten <=0cm,->,transform canvas={xshift=+2mm}] (m-6-3)  to[Rshifttight] node[pos=0.75,right,scale=0.6]{$\adjL$} (m-1-3);

    \draw[shorten >=0.12cm,shorten <=0.12cm,white,line width=5mm] (m-5-5) -- (m-10-5);
    \draw[->] (m-5-5) to[Rshifttight] node[pos=0.75,left,scale=0.6]{$\adjR$} (m-10-5);
    \draw[->] (m-10-5) to[Rshifttight] node[pos=0.25,right,scale=0.6]{$\adjL$} (m-5-5);

    \draw[shorten >=0.12cm,shorten <=0.12cm,white,line width=5mm] (m-3-4) -- (m-5-1);
    \draw[shorten >=0.7cm,shorten <=0.7cm,->] (m-3-4) to[Lshifttight] node[below,scale=0.6]{$\adjL$} (m-5-1);
    \draw[shorten >=0.7cm,shorten <=0.7cm,->] (m-5-1) to[Lshifttight] node[above,scale=0.6]{$\adjR$} (m-3-4);

    \node[scale=0.9,fill=white] at (m-1-3) {$(\wt{\calC},?,\dg\wt{\calF}\cap\dw\calW\cap\calF)$};
    \node[scale=0.9,fill=white] at (m-1-7) {$(\dg\wt{\calC},\Acyc(\scA),\dg\wt{\calW\cap\calF})$};
    \node[scale=0.9,fill=white] at (m-3-4) {$(\dg\wt{\calC},?,\dg\wt{\calF}\cap\dw\calW\cap\calF)$};
    \node[scale=0.9,fill=white] at (m-5-1) {$(\dg\wt{\calC},?,\wt{\calF}\cap\dw\calW\cap\calF)$};
    \node[scale=0.9,fill=white] at ($(m-5-5) - (0mm,0)$) {$({^{\perp}}[\wt{\calF}\cap\dw\calW\cap\calF],\Acyc(\scA),\dg\wt{\calF}\cap\dw\calW\cap\calF)$};
    \node[scale=0.9,fill=white] at (m-6-3) {$(\wt{\calC}\cap\dw\calC\cap\calW,?,\dg\wt{\calF})$};
    \node[scale=0.9,fill=white] at ($(m-6-7) - (5mm,0)$) {$(\dg\wt{\calC}\cap\dw\calC\cap\calW,\Acyc(\scA),[\wt{\calC}\cap\dw\calC\cap\calW]^{\perp})$};
    \node[scale=0.9,fill=white] at (m-8-4) {$(\dg\wt{\calC}\cap\dw\calC\cap\calW,?,\dg\wt{\calF})$};
    \node[scale=0.9,fill=white] at (m-10-1) {$(\dg\wt{\calC}\cap\dw\calC\cap\calW,?,\wt{\calF})$};
    \node[scale=0.9,fill=white] at (m-10-5) {$(\dg\wt{\calC\cap\calW},\Acyc(\scA),\dg\wt{\calF})$};
\end{tikzpicture}\end{equation*}
\caption{Abelian model structures on $\Ch(\scA)$ induced by an abelian model structure on a Grothendieck abelian category $\scA$}\label{fig_chainbutterfly}
\end{figure}

Finally, the induced functor $\bfD(\scA)\to \Ho(\calM)$ solves the realization problem:

\begin{maintheorem}\label{thm_leftrightstabilization}
In the situation of Theorem \ref{thm_chainbutterfly}, the composition
$$\bfD(\scA)\cong\Ho(\dg\wt{\calC\cap\calW},\Acyc(\scA),\dg\wt{\calF})\longrightarrow\Ho(\wt{\calC}\cap\dw\calC\cap\calW,?,\dg\wt{\calF})\cong\Ho(\calM)$$
of the left stabilization functor associated to the lower butterfly in Figure \ref{fig_chainbutterfly} and the
equivalence from Theorem \ref{thm_chainbutterfly}.\eqref{thm_chainbutterfly_it3} makes the following diagram commutative,
\begin{equation*}\label{eq:realizationsolution}\begin{tikzpicture}[baseline,description/.style={fill=white,inner sep=2pt}]
    \matrix (m) [matrix of math nodes, row sep=1.5em,
                 column sep=2.5em, text height=1.5ex, text depth=0.25ex,
                 inner sep=0pt, nodes={inner xsep=0.3333em, inner ysep=0.3333em}]
    {
       \Ho(\calM) && \bfD(\scA)\\
       & \scA, & \\
    };
    \draw[->,rounded corners=4mm] (m-2-2) -| (m-1-3);
    \draw[->,rounded corners=4mm] (m-2-2) -| (m-1-1);
    \draw[->] (m-1-3) -- (m-1-1);
\end{tikzpicture}\end{equation*}
thereby solving the realization problem for $\calM$.
\end{maintheorem}

\section{Proof of Theorem \ref{thm_bunch}}

We begin by recalling the definitions of some crucial classes of complexes; see e.g. \cite[Definition 3.3]{Gillespie_FlatModelStructure}, \cite[Notation 4.1]{StovicekHillGrothendieck}:

\begin{definition}\label{def_classes}
Let $(\calC,\calD)$ be a cotorsion pair in the abelian category $\scA$.
\begin{enumerate}
\item $\dw\calC := \Ch(\calC)$ and $\dw\calD := \Ch(\calD)$.
\item $\wt{\calC}$ denotes the class of acyclic complexes in $\scA$ with syzygies in $\calC$.
\item $\wt{\calD}$ denotes the class of acyclic complexes in $\scA$ with syzygies in $\calD$.
\item $\dg\wt{\calC} := {^{\perp}}\wt{\calD}$ and $\dg\wt{\calD} :=
  \wt{\calC}^{\perp}$.
\end{enumerate}
\end{definition}
\begin{example}
Considering the cotorsion pairs $(\proj,\scA)$ and $(\scA,\inj)$ one recovers the classes of dg projective and
dg injective complexes as $\dg\proj$ and $\dg\inj$, respectively.
\end{example}
\begin{proposition}\label{prop_classes}
Let $(\calC,\calD)$ be a cotorsion pair over $\scA$ and $X\in\Ch(\scA)$.
\begin{enumerate}[leftmargin=1.25cm,ref=(\ref{prop_classes}.\roman*)]
\item\label{item:propclassesdgC} $X\in\dg\calC$ if and only if $X\in\dw\calC$ and $[X,D]=0$ for all $D\in\wt{\calD}$.
\item\label{item:propclassesdgD} $X\in\dg\calD$ if and only if $X\in\dw\calD$ and $[C,X]=0$ for all $C\in\wt{\calC}$.
\end{enumerate}
Further, we have the following inclusions:
\begin{enumerate}
\item[(iii)]\label{item:propclassesdg} $\Ch^-(\calC)\subset\dg\wt{\calC}$ and $\Ch^+(\calD)\subset\dg\wt{\calD}$.
\item[(iv)]\label{item:propclassesacyc} $\Ch^+(\scA)\cap\wt{\calC}\subset {^{\perp}}\left[\dw\calD\right]$ and $\Ch^-(\scA)\cap\wt{\calD}\subset\left[\dw\calC\right]^{\perp}$.\item[(v)]\label{item:propclasseswt} $\wt{\calC}\subset\dg\wt{\calC}$ and $\wt{\calD}\subset\dg\wt{\calD}$.
\end{enumerate}
\end{proposition}
\begin{proof}
This is mostly contained in \cite[\S 3]{Gillespie_FlatModelStructure}, but for convenience we include an argument here. First,
the (exact) adjoints $G^{\pm}: \scA^\ZZ\to\Ch(\scA)$ to the (exact) forgetful functor $(-)\us:
\Ch(\scA)\to\scA^\ZZ$ map $\calC^\ZZ$ resp. $\calD^\ZZ$ to $\wt{\calC}$
resp. $\wt{\calD}$. Hence, given $X\in\dg\wt{\calC}$ and $D\in\calD^\ZZ$, we have
$0=\ext^1_{\Ch(\scA)}(X,G^-(D))\cong\ext^1_{\scA^\ZZ}(X\us,D)$, so $X\us\in{^{\perp}}(\calD^\ZZ)=\calC^\ZZ$,
i.e. $X\in\dw\calC$. Similarly, we have $\dg\wt{\calD}\subset\dw\calD$.

Next, if $X\in\dw\calD$ and $Z\in\dw\calC$, then any short exact sequence $0\to X\to Y\to Z\to 0$ in $\Ch(\scA)$ is
degree-wise split, so that $\ext^1_{\Ch(\scA)}(Z,X)$ is canonically isomorphic to the extension group
$\ext^1_{\dw\Ch(\scA)}(Z,X)\cong [Z,\Sigma X]$ with respect to the degree-wise split exact structure on
$\Ch(\scA)$. This proves the first two claims (i) and (ii).

The proof of the inclusions in (iii) and (iv) is analogous to
the proof of the classical fact that chain maps from bounded above complexes of projectives to acyclic complexes are
nullhomotopic, as are chain maps from acyclic complexes to bounded below complexes of injectives. Finally, the inclusions
$\wt{\calC}\subset\dg\wt{\calC}$ and $\wt{\calD}\subset\dg\wt{\calD}$ from part (v) are proved in
\cite[Lemma 3.9]{Gillespie_FlatModelStructure}.
\end{proof}
The following beautiful theorem is the result of long work by Gillespie \cite{Gillespie_FlatModelStructure,
  Gillespie_FlatModelStructureSheaves} in his studies of the flat model structures on $\Ch(R\Mod)$ and $\Ch(\calO_X)$ (for a
survey, see \cite[\S 7]{HoveyCotorsionArxiv}) and new results of Stovicek \cite{StovicekHillGrothendieck} on
deconstructible classes in Grothendieck categories. To be precise, \cite[Proposition 3.6,
Corollary 3.7]{Gillespie_FlatModelStructureSheaves} essentially prove parts (i) and (iii), while the completeness of
$(\wt{\calC},\dg\wt{\calD})$ crucial for part (ii) is guaranteed by the deconstructibility of $\wt{\calC}$ established
in \cite[Theorem 4.2]{StovicekHillGrothendieck}. We collect these arguments and give a proof for
convenience of the reader.
\hyphenation{he-re-di-ta-ry}
\begin{theorem}[\protect{\cite{Gillespie_FlatModelStructureSheaves,StovicekHillGrothendieck}}]\label{theorem_gillespiestovicek}
Let $\scA$ be a Grothendieck category and $(\calC,\calD)$ be a small and hereditary cotorsion pair in $\scA$. Then the
following hold:
\begin{enumerate}
\item $(\dg\wt{\calC},\wt{\calD})$ is a small, hereditary cotorsion pairs in $\Ch(\scA)$.
\item $(\wt{\calC},\dg\wt{\calD})$ is a small, hereditary cotorsion pairs in $\Ch(\scA)$.
\item $\wt{\calC}=\dg\wt{\calC}\cap\Acyc(\scA)$ and $\wt{\calD}=\dg\wt{\calD}\cap\Acyc(\scA)$.
\end{enumerate}
\end{theorem}
\begin{corollary}\label{cor_gillespiestovicek}
In the situation of Theorem \ref{theorem_gillespiestovicek}, $(\dg\wt{\calC},\Acyc(\scA),\dg\wt{\calD})$ is a cofibrantly
generated abelian model structure on $\Ch(\scA)$ with homotopy category $\bfD(\scA)$.
\end{corollary}
\begin{example}
Applying Corollary \ref{cor_gillespiestovicek} to the cotorsion pair $(\scA,\inj)$ gives rise to the injective
model $(\Ch(\scA),\Acyc(\scA),\dg\inj)$ for $\bfD(\scA)$. Similarly, if $\scA$ has enough projectives the
cotorsion pair $(\proj,\scA)$ yields the projective model $(\dg\proj,\Acyc(\scA),\Ch(\scA))$ for
$\bfD(\scA)$. A nontrivial example -- and in fact \emph{the} example that started the theory -- is obtained from the
flat cotorsion pair $(\flatmod(R),\flatmod(R)^{\perp})$ on $R\Mod$, with $R$ any ring:
in this case, one obtains Gillespie's flat model structure.
\end{example}
\begin{proof}[Proof of Theorem \ref{theorem_gillespiestovicek}]
We first prove (i). To begin, \cite[Proposition
3.6]{Gillespie_FlatModelStructure} shows that $(\dg\wt{\calC},\wt{\calD})$ and
$(\wt{\calC},\dg\wt{\calD})$ are cotorsion pairs; in addition to $(\calC,\calD)$ being a cotorsion pair,
this only needs the assumption that $\calC$ is generating and $\calD$ is cogenerating in $\scA$.

Concerning the smallness of $(\dg\wt{\calC},\wt{\calD})$ and $(\wt{\calC},\dg\wt{\calD})$, applying \cite[Theorem
4.2]{StovicekHillGrothendieck} (generalizing the ideas used by Gillespie \cite{Gillespie_FlatModelStructure} in the case
of the flat model structure on $\Ch(R\Mod)$) shows that $\dg\wt{\calC}$ and $\wt{\calC}$ are deconstructible, so it remains to
check that $\wt{\calC}$ is generating. For this, note that since $\calC$ is generating in $\scA$, $\calC^\ZZ$ is
generating in $\scA^\ZZ$; the counit $G^+(X\us)\to X$ being an epimorphism for $X\in\Ch(\scA)$, it follows that
$G^+(\calC^\ZZ)$ is generating in $\Ch(\scA)$. We have $G^+(\calC^\ZZ)\subset\wt{\calC}$, so $\wt{\calC}$ is
generating, too.

To check that $(\dg\wt{\calC},\wt{\calD})$ and $(\wt{\calC},\dg\wt{\calD})$ are hereditary, it suffices (by
\cite[Corollary 1.1.12]{Becker_ModelsForSingularityCategories}) to show that $\dg\wt{\calC}$ is resolving while
$\dg\wt{\calD}$ is coresolving. We only check that
$\dg\wt{\calC}$ is resolving, the proof of $\dg\wt{\calD}$ being analogous. For that, recall from Proposition
\ref{prop_classes} that $\dg\wt{\calC}$ consists of those $X\in\dw\calC$ for which $\Hom\ua_\scA(X,D)\in\Acyc(\ZZ)$ for
all $D\in\wt{\calD}$, and suppose $0\to X\to Y\to Z\to 0$ is a short exact sequence in $\Ch(\scA)$ with
$Y,Z\in\dg\wt{\calC}$. Then, firstly, $X\in\dw\calC$ since $\calC$ is resolving and $Y,Z\in\dw\calC$. Further, for any
$D\in\wt{\calD}$ (even any $D\in\dw\calD$), the sequence of complexes of abelian groups
$$0\to\Hom\ua_\scA(Z,D)\to\Hom\ua_\scA(Y,D)\to\Hom\ua_\scA(X,D)\to 0$$ is exact, and since $\Hom\ua_\scA(Z,D)$ and
$\Hom\ua_\scA(Y,D)$ are exact by our assumption that $Y,Z\in\dg\wt{\calC}$, it follows that $\Hom\ua_\scA(X,D)$ is
exact, too.

Concerning (ii), \cite[Theorem 3.12]{Gillespie_FlatModelStructure} shows that $\dg\wt{\calC}\cap\Acyc=\wt{\calC}$, and
in view of \cite[Lemma 3.14(1)]{Gillespie_FlatModelStructure} and the completeness of $(\dg\wt{\calC},\wt{\calD})$ we
already proved, this also shows $\dg\wt{\calD}\cap\Acyc=\wt{\calD}$.
\end{proof}

Suppose now that $\calM=(\calC,\calW,\calF)$ is a cofibrantly generated and hereditary abelian model structure on the
Grothendieck category $\scA$, and put $\omega := \calC\cap\calW\cap\calF$, the \emph{core} of $\calM$. We will be
concerned with quite a number of induced cotorsion pairs all of which will have
$\wt{\calC}\cap\wt{\calW\cap\calF}=\wt{\calC\cap\calW}\cap\wt{\calF}=:\wt{\omega}$ as their core, the class of acyclic
complexes with syzygies in $\omega$. In view of the following lemma, these are precisely the contractible complexes with
entries in $\omega$:
\begin{lemma}\label{lem_contractible}
For $\scA$ abelian and $X\in\Acyc(\scA)$, the following are equivalent:
\begin{enumerate}
\item $X$ is contractible.
\item The exact sequences $0\to \syz^n X\to X\to \syz^{n+1} X\to 0$ split.
\end{enumerate}
\end{lemma}
\begin{proof} Omitted.\end{proof}
\begin{lemma}\label{lem_doubleorthogonalagain}
Let $\scA$ be an abelian category and $\omega$ be a self-orthogonal class of objects in $\scA$,
i.e. $\omega\subset{^{\perp}}\omega$. Then $\wt{\omega}$, the class of contractible complexes with entries in $\omega$,
is the largest self-orthogonal, $\Sigma$-stable class in $\Ch(\scA)$ contained in $\dw\omega$.
\end{lemma}
\begin{proof}
By self-orthogonality of $\omega$, any short exact sequence $0\to X\to Y\to Z\to 0$ with $X,Y\in\dw\omega$ is
degree-wise split, and hence determined by a homotopy class in $[Z,\Sigma X]$. This shows that
$\wt{\omega}$ is self-orthogonal in $\Ch(\scA)$.

Conversely, suppose $\calE\subset\dw\omega$ is self-orthogonal and $\Sigma$-stable,
i.e. $\Sigma\calE\subset\calE$. Then, given any $X\in\calE$ we have $0=\ext^1_{\Ch(\scA)}(\Sigma
X,X)=\ext^1_{\dw\Ch(\scA)}(\Sigma X,X)\cong [\Sigma X,\Sigma X]$, so $X$ is contractible.
\end{proof}

\begin{proof}[Proof of Theorem \ref{thm_bunch}]
We begin by showing that all cotorsion pairs are small.

For the top square, it suffices to check that the right hand sides
of the cotorsion pairs listed in it are of the form $\calS^{\perp}$ for a generating set $\calS\subset\Ch(\scA)$. This
property is preserved under intersection, so we need to check it for $\wt{\calF}$, $\dg\wt{\calF}$ and
$\dw\calW\cap\calF$ only. Concerning the first two, we know from Theorem \ref{theorem_gillespiestovicek} that
$\wt{\calF}=\left[\dg\wt{\calC\cap\calW}\right]^{\perp}$ and $\dg\wt{\calF}=\left[\wt{\calC\cap\calW}\right]^{\perp}$, with $\wt{\calC\cap\calW}$
and $\dg\wt{\calC\cap\calW}$ both deconstructible and generating. For $\dw\calW\cap\calF$, note that $X\in\dw\calW\cap\calF$ if
and only if for all $C\in\calC^\ZZ$  we have $0=\ext^1_{\scA^\ZZ}(C,X\us)\cong\ext^1_{\Ch(\scA)}(G^+(C),X)$, so that by
cocontinuity and exactness of $G^+$ we conclude that $\dw\calW\cap\calF=G^+(\calS)^{\perp}$ for $\calS\subset\scA$ some
set chosen such
that $\calC=\filt\calS$; as $\calC$ is generating, we may assume $\calS$ generating, too, and then $G^+(\calS)$ is
generating in $\Ch(\scA)$ since the counit $G^+ X\us\to X$ is an epimorphism for all $X\in\Ch(\scA)$. This concludes the
proof that all cotorsion pairs in the upper square are small.

For the middle square, all cotorsion pairs contained in it are of the form studied in Theorem \ref{theorem_gillespiestovicek}, hence
small. Finally, to prove that the cotorsion pairs in the lower square are small it suffices to show that their left hand
sides are generating and deconstructible. They are generating as they all contain the generating class
$G^+(\calC\cap\calW)$, and deconstructibility follows from the stability of deconstructible classes under intersection
\cite[Proposition 2.9]{StovicekHillGrothendieck} as well as the deconstructibility of $\wt{\calC}$, $\dg\wt{\calC}$
and $\dw\wt{\calC\cap\calW}$ \cite[Theorem 4.2]{StovicekHillGrothendieck}.

Next we check that all cotorsion pairs are hereditary. For the ones in middle square, this follows from Theorem
\ref{theorem_gillespiestovicek} above. Concerning the ones in the upper square, their right hand sides are coresolving
as intersections of the classes $\wt{\calF}$, $\dg\wt{\calF}$ and $\dw\calW\cap\calF$, each of which is coresolving:
the first two are again treated as part of Theorem \ref{theorem_gillespiestovicek}, while $\dw\calW\cap\calF$ is
coresolving since $\calW\cap\calF$ is. Applying \cite[Corollary 1.1.12]{Becker_ModelsForSingularityCategories} then
shows that all cotorsion pairs in the upper square are hereditary, and the reasoning for the lower square is analogous.

Finally we check that all cotorsion pairs have core equal to $\wt{\omega}$, the class of contractible complexes with
values in $\calC\cap\calW\cap\calF$. First, using the fact $\ext^1_{\dw\Ch\scA}(-,-)\cong [\Sigma(-),-]$ it
is a quick check that $\wt{\omega}$ is contained in all the cores. For the reverse inclusion, Lemma
\ref{lem_doubleorthogonalagain} and the stability under shift of all the classes involved show that it suffices to check
that all cores are contained in $\dw\omega$. For the middle square, this is clear. For the upper square, all the
cotorsion pairs in it have their right hand sides contained in $\dw\calW\cap\calF$, and the fact that they are all
connected to a cotorsion pair in the middle row by a chain of arrows shows that their left hand sides are all contained in
$\dw\calC$. Similarly, the left hand sides of the cotorsion pairs in the lower square are all contained in
$\dw\calC\cap\calW$, while all of them receiving an arrow from the middle square shows that their right hand sides are
all contained in $\dw\calF$.
\end{proof}

\section{Proof of Theorem \ref{thm_chainbutterfly}}

We now elaborate on the following statements, leading to the proof of Theorem \ref{thm_chainbutterfly}:
\begin{enumerate}[leftmargin=1cm,label=(\roman*)]
\item\label{item:Btriangle} Each triangle $\ortriangle{0.4}{1mm}$ in Figure \ref{fig_bunch} gives rise to a localization sequence between the three
  model structures induced by its edges (Proposition \ref{prop_locconcomp}).
\item\label{item:Bsquare} Each square $\orsquare{0.4}{1mm}$ in Figure \ref{fig_bunch} yields a butterfly-shaped diagram of adjunctions
  between the two localization sequences associated via \ref{item:Btriangle} to its triangle faces. However,
  these are not necessarily butterflies.
\item\label{item:Bdashed} The dashed arrows in Figure \ref{fig_bunch} yield four models for $\bfD(\scA)$ (Proposition \ref{prop_bunchderived}).
\item\label{item:Bsnaked} The snaked arrows yield four models structures on $\Ch(\scA)$ Quillen equivalent to our given model
  structure $(\calC,\calW,\calF)$ on $\scA$ (Theorem \ref{thm_lifttocomplexes}).
\item\label{item:Bzigzag} The zigzag arrows induce model structures analogous to the co- and contraderived model structures,
  and the upper right and lower left triangles induce localization sequences connecting them to the models for $\bfD(\scA)$
  from \ref{item:Bdashed} and the model structures associated to the dashdotted arrows. See Example
  \ref{ex_coctrgeneral}.
\item\label{item:Btilted} Each of the two tilted squares in Figure \ref{fig_bunch} composed out of two dashed, two snaked and
  one dotted arrow, gives rise to a butterfly between the model structures associated to the dotted arrows and the
  models from \ref{item:Bdashed} and \ref{item:Bsnaked}.
\end{enumerate}

We begin with the model structures induced by the snaked arrows in Figure \ref{fig_bunch}:

\begin{theorem}\label{thm_lifttocomplexes}
Let $\calM=(\calC,\calW,\calF)$ be a cofibrantly generated, hereditary abelian model structure on the Grothendieck
category $\scA$. Then there is a square of cofibrantly generated abelian model structures on $\Ch(\scA)$ and identity
Quillen equivalences:
\begin{equation}\label{eq:liftsquare}\begin{tikzpicture}[baseline,description/.style={fill=white,inner sep=2pt}]
    \matrix (m) [matrix of math nodes, row sep=2em,
                 column sep=2.5em, text height=1.5ex, text depth=0.25ex,
                 inner sep=0pt, nodes={inner xsep=0.3333em, inner ysep=0.3333em}]
    {
       (\wt{\calC},?,\dg\wt{\calF}\cap\dw\calW\cap\calF) & (\dg\wt{\calC},?,\wt{\calF}\cap\dw\calW\cap\calF)\\
       (\wt{\calC}\cap\dw\calC\cap\calW,?,\dg\wt{\calF}) & (\dg\wt{\calC}\cap\dw\calC\cap\calW,?,\wt{\calF})\\
    };
    \draw[->,transform canvas={yshift=0.8mm}]  (m-1-1) -- node[above,scale=0.75]{$\adjL$} (m-1-2);
    \draw[->,transform canvas={yshift=-0.8mm}]  (m-1-2) -- node[below,scale=0.75]{$\adjR$} (m-1-1);
    \draw[->,transform canvas={yshift=0.8mm}]  (m-2-1) -- node[above,scale=0.75]{$\adjL$} (m-2-2);
    \draw[->,transform canvas={yshift=-0.8mm}]  (m-2-2) -- node[below,scale=0.75]{$\adjR$} (m-2-1);
    \draw[->,transform canvas={xshift=-0.8mm}]  (m-1-1) -- node[left,scale=0.75]{$\adjR$} (m-2-1);
    \draw[->,transform canvas={xshift=+0.8mm}]  (m-2-1) -- node[right,scale=0.75]{$\adjL$} (m-1-1);
    \draw[->,transform canvas={xshift=-0.8mm}]  (m-1-2) -- node[left,scale=0.75]{$\adjR$} (m-2-2);
    \draw[->,transform canvas={xshift=+0.8mm}]  (m-2-2) -- node[right,scale=0.75]{$\adjL$} (m-1-2);
\end{tikzpicture}\end{equation}
Their homotopy categories are equivalent to the homotopy category of acyclic complexes with entries in
$\calC\cap\calW\cap\calF$ and syzygies in $\calC\cap\calF$. Moreover, there are Quillen equivalences
\begin{align}\label{eq:frobeniuslift1}
\cosyz^0: (\wt{\calC}\cap\dw\calC\cap\calW,?,\dg\wt{\calF})\rightleftarrows(\calC,\calW,\calF): \iota^0\\\label{eq:frobeniuslift2}
\iota^0: (\calC,\calW,\calF)\rightleftarrows (\dg\wt{\calC},?,\wt{\calF}\cap\dw\calW\cap\calF): \syz^0
\end{align}
which on the homotopy categories yield the classical equivalences from Theorem \ref{theorem_frobeniusclassical}.
\end{theorem}
In particular, one has to beware that the derived adjoint equivalences of \eqref{eq:frobeniuslift1} and
\eqref{eq:frobeniuslift2} are \emph{not} isomorphic, but are shifts of one another.
\begin{example}\label{ex_liftgorenstein}
Suppose $R$ is a Gorenstein ring and consider Hovey's Gorenstein projective model structure
$\calM\ugproj(R)=(\gproj(R),\calP^{<\infty}(R),R\Mod)$ on $R\Mod$. In this case, we recover some constructions and
results of \cite[Sections 2, 3.1]{Becker_ModelsForSingularityCategories}: Namely,
$(\wt{\calC}\cap\dw\calC\cap\calW,?,\dg\wt{\calF})$ coincides with the projective singular contraderived model structure
${^{p}}\calM\uctr\lsing(R):=(\Ch(\Proj(R)),?,\Ch(R))$ (see \cite[Proposition
2.2.1]{Becker_ModelsForSingularityCategories}) since the syzygies of any acyclic complex of projectives are
automatically Gorenstein projective. The equivalence
\eqref{eq:frobeniuslift2} therefore agrees with the Quillen equivalence obtained in \cite[Proposition
3.1.3]{Becker_ModelsForSingularityCategories}. Also, we claim that
$(\dg\wt{\calC}\cap\dw\calC\cap\calW,?,\wt{\calF})$ coincides with the
classical singular contraderived model structure $\calM\lsing\uctr(R):=(\Ch(\Proj(R)),?,\Acyc(R))$ (see \cite[Definition
2.1.2]{Becker_ModelsForSingularityCategories}), i.e. that
$\Ch(\Proj(R))\subseteq{^{\perp}}\wt{\calI^{<\infty}}$: For this, note first that for $P\in\Ch(\Proj(R))$ the short exact
sequence $0\to P\to G^+(\Sigma P\us)\to \Sigma P\to 0$ exhibits $P$ as the syzygy of $\Sigma P$ in the abelian
category $\Ch(R)$, since $G^+(\Sigma P\us)$ is projective in $\Ch(R)$; see \cite[Section
1.3]{Becker_ModelsForSingularityCategories} for the relevant results and notions. Iterating this procedure, we see that
any $P\in\Ch(\Proj(R))$ is an arbitrarily high syzygy in $\Ch(R)$. On the other hand, any
complex in $\wt{\calI^{<\infty}}$ admits a finite resolution by contractible complexes of injectives, i.e. has finite
injective dimension in the abelian category $\Ch(R)$. The claim $\Ch(\Proj(R))\subseteq{^{\perp}}\wt{\calI^{<\infty}}$
follows.
\end{example}
\begin{proof}[Proof of Theorem \ref{thm_lifttocomplexes}]
It only remains to prove that \eqref{eq:frobeniuslift1} and \eqref{eq:frobeniuslift2} are Quillen equivalences. We start
by checking that \eqref{eq:frobeniuslift1} is a Quillen adjunction. Suppose $\iota: X\to Y$ is a cofibration in
$(\wt{\calC}\cap\dw\calC\cap\calW,?,\dg\wt{\calF})$, i.e. $\iota$ is a monomorphism in $\Ch(\scA)$ with cokernel $Z :=
\coker(\iota)$ belonging to $\wt{\calC}\cap\dw\calC\cap\calW$. Since $\wt{\calC}\subset\Acyc(\scA)$, it follows as in
\cite[Proposition 3.1.3]{Becker_ModelsForSingularityCategories} that $0\to \cosyz^0 X\to \cosyz^0 Y\to \cosyz^0 Z\to 0$ is exact in
$\scA$, and as $\cosyz^0 Z\in\calC$ by definition, it follows that $\cosyz^0\iota$ is a cofibration in $(\calC,\calW,\calF)$. If
$\iota$ is a trivial cofibration,
then $Z\in{^{\perp}}\dg\wt{\calF}=\wt{\calC\cap\calW}$, hence exact, and we deduce an exact sequence $0\to \cosyz^0 X\to \cosyz^0
Y\to \cosyz^0 Z\to 0$ with $\cosyz^0 Z\in \calC\cap\calW$. This shows that \eqref{eq:frobeniuslift1} is a Quillen adjunction, and for
\eqref{eq:frobeniuslift2} the proof is analogous.

Next we prove that \eqref{eq:frobeniuslift1} is a Quillen equivalence. In the one direction, consider a bifibrant
$X$ in $(\wt{\calC}\cap\dw\calC\cap\calW,?,\dg\wt{\calF})$, that is,
$X\in\wt{\calC}\cap\dg\wt{\calF}\cap\dw\calC\cap\calW$. Since $\dg\wt{\calF}\cap\Acyc(\scA)=\wt{\calF}$, we have
$X\in\wt{\calC}\cap\wt{\calF}\cap\dw\calC\cap\calW$, and in particular $\cosyz^0 X\in\calF$ is fibrant. Hence, to show the
derived unit is an equivalence, it suffices to show that such an $X$, the underived unit $\veps:
X\to \iota^0 \cosyz^0(X)$ is a weak equivalence in $(\wt{\calC}\cap\dw\calC\cap\calW,?,\dg\wt{\calF})$. We have that $\veps$
is an epimorphism and $\ker(\veps) \cong
\tau_{\leq 0}X\oplus\sigma_{>0}X$, and we consider the two summands separately. The first summand $\tau_{\leq 0} X$ belongs
to $\Ch^-(\scA)\cap\wt{\calF}$ which is contained in
$\left[\dw\calC\cap\calW\right]^{\perp}$ by Proposition \ref{prop_classes}(iv), hence trivially fibrant in
$(\wt{\calC}\cap\dw\calC\cap\calW,?,\dg\wt{\calF})$. The second summand $\sigma_{>0}(X)$ belongs to
$\Ch^+(\calW\cap\calF)$ which is contained in $\wt{\calC}^{\perp}$ by Proposition \ref{prop_classes}(iii), hence
trivially fibrant, too. It follows that $\veps: X\to \iota^0 \cosyz^0 (X)$ is indeed weak equivalence in
$(\wt{\calC}\cap\dw\calC\cap\calW,?,\dg\wt{\calF})$.

We have just proved that the derived unit $\id\Rightarrow\bfR\iota^0\circ\bfL \cosyz^0$ is an equivalence,
which means that $\bfL \cosyz^0$ is fully faithful. To prove that $\bfL \cosyz^0\dashv\bfR\iota^0$ is an equivalence, it is
therefore enough to show that $\bfL \cosyz^0$ is also essentially surjective. For this, it suffices to check that any
bifibrant $M\in\calC\cap\calF$ occurs as the $0$-th syzygy $\cosyz^0 X$ of some bifibrant ``complete resolution'' $X\in
\wt{\calC}\cap\wt{\calF}\cap\dw\calC\cap\calW$. Such a resolution can be built inductively using the completeness of the
cotorsion pairs $(\calC\cap\calW,\calF)$ and $(\calC,\calW\cap\calF)$.

The proof that \eqref{eq:frobeniuslift2} is a Quillen equivalence is analogous.
\end{proof}
\begin{corollary}
Any hereditary and cofibrantly generated abelian model structure on a Grothendieck category is Quillen equivalent to
an abelian model structure on $\Ch(\scA)$.
\end{corollary}

Next we study the model structures induced by the dashed arrows in Figure \ref{fig_bunch}.

\begin{proposition}\label{prop_bunchderived}
Let $\calM=(\calC,\calW,\calF)$ be a cofibrantly generated, hereditary abelian model structure on the Grothendieck
category $\scA$. Then there is a square of cofibrantly generated abelian model structures on $\Ch(\scA)$ and identity
Quillen equivalences:
\begin{equation}\label{eq:liftsquarederived}\begin{tikzpicture}[baseline,description/.style={fill=white,inner sep=2pt}]
    \node (m-1-1) at (-3.5cm,0) {$(\dg\wt{\calC\cap\calW},\Acyc(\scA),\dg\wt{\calF})$};
    \node (m-1-2) at (3cm,0)
    {$(\dg\wt{\calC}\cap\dw\calC\cap\calW,\Acyc(\scA),(\wt{\calC}\cap\dw\calC\cap\calW)^{\perp})$};
    \node (m-2-1) at (-2cm,-4em)
    {$({^{\perp}}(\wt{\calF}\cap\dw\calW\cap\calF),\Acyc(\scA),\dg\wt{\calF}\cap\dw\calW\cap\calF)$};
    \node (m-2-2) at (4.5cm,-4em) {$(\dg\wt{\calC},\Acyc(\scA),\dg\wt{\calW\cap\calF})$};
    \draw[->,transform canvas={yshift=0.8mm}]  (m-1-1) -- node[above,scale=0.75]{$\adjL$} (m-1-2);
    \draw[->,transform canvas={yshift=-0.8mm}]  (m-1-2) -- node[below,scale=0.75]{$\adjR$} (m-1-1);
    \draw[->,transform canvas={yshift=0.8mm}]  (m-2-1) -- node[above,scale=0.75]{$\adjL$} (m-2-2);
    \draw[->,transform canvas={yshift=-0.8mm}]  (m-2-2) -- node[below,scale=0.75]{$\adjR$} (m-2-1);
    \draw[->,transform canvas={xshift=-0.8mm}]  (m-1-1) -- node[left,scale=0.75]{$\adjL$} ($(m-1-1 |- m-2-1) +
    (0,3mm)$);
    \draw[->,transform canvas={xshift=+0.8mm}]  ($(m-1-1 |- m-2-1) + (0,3mm)$) -- node[right,scale=0.75]{$\adjR$}
    (m-1-1);
    \draw[->,transform canvas={xshift=-0.8mm}]  (m-2-2) -- node[left,scale=0.75]{$\adjR$} ($(m-2-2 |- m-1-2) +
    (0,-3mm)$);
    \draw[->,transform canvas={xshift=+0.8mm}]  ($(m-2-2 |- m-1-2) + (0,-3mm)$) -- node[right,scale=0.75]{$\adjL$} (m-2-2);
\end{tikzpicture}\end{equation}
Their homotopy categories are equivalent to the ordinary derived category $\bfD(\scA)$.
\end{proposition}
\begin{proof}
Applying Gillespie's Theorem \ref{thm_gillespie} to the dashed arrows in Figure \ref{fig_bunch} gives four model structures
matching the triples listed in \eqref{eq:liftsquarederived} in the left and right hand parts; it therefore suffices to check
that their classes of weakly trivial objects all coincide with the class $\Acyc(\scA)$ of acyclic complexes.

For the model structures associated to the arrows
$(\wt{\calC\cap\calW},\dg\wt{\calF})\to(\dg\wt{\calC\cap\calW},\wt{\calF})$ and
$(\wt{\calC},\dg\wt{\calW\cap\calF})\to(\dg\wt{\calC},\dg\wt{\calW\cap\calF})$ we already know this from Corollary
\ref{cor_gillespiestovicek} above. Next, consider model structure associated to
$(\wt{\calC}\cap\dw\calC\cap\calW,\perp)\to(\dg\wt{\calC}\cap\dw\calC\cap\calW,\perp)$: By \eqref{eq:thmgillespiew}
the weakly trivial objects in the associated model structure are those $X\in\Ch(\scA)$ which admit a short exact
sequence of the form $0\to F\to C\to X\to 0$ with $F\in\left[\dg\wt{\calC}\cap\dw\calC\cap\calW\right]^{\perp}\subset\wt{\calF}$ and
$C\in\wt{\calC}\cap\dw\calC\cap\calW$. Note that $F,C\in\Acyc(\scA)$, so the existence of such a sequence implies that
$X\in\Acyc(\scA)$. Conversely, suppose $X\in\Acyc(\scA)$ and pick an approximation sequence $0\to F\to C\to X\to 0$ for
the cotorsion pair $(\dg\wt{\calC}\cap\dw\calC\cap\calW,\left[\dg\wt{\calC}\cap\dw\calC\cap\calW\right]^{\perp})$. Then
again $F\in\Acyc(\scA)$, and also $X\in\Acyc(\scA)$ by assumption, so $C\in
\dg\wt{\calC}\cap\Acyc(\scA)\cap\dw\calC\cap\calW=\wt{\calC}\cap\dw\calC\cap\calW$.

The proof that the weak equivalences in the model structure associated to the arrow
$(\perp,\dg\wt{\calF}\cap\dw\calW\cap\calF)\to(\perp,\wt{\calF}\cap\dw\calW\cap\calF)$ in Figure \ref{fig_bunch} is analogous.
\end{proof}

Finally, we study the relation between the model structures induced by the tilted squares in Figure \ref{fig_bunch}, beginning
with some observations that hold in general:

\begin{proposition}\label{prop_locconcomp}
For localization contexts $(\calD\pp,\calE\pp)\locarrow{\alpha}(\calD\p,\calE\p)\locarrow{\beta}(\calD,\calE)$ their
induced model structures are related  via a localization sequence of triangulated categories:
\begin{equation}\label{eq:localizationcomposition}\begin{tikzpicture}[baseline,description/.style={fill=white,inner sep=2pt}]
    \matrix (m) [matrix of math nodes, row sep=0.75em,
                 column sep=5.5em, text height=1.5ex, text depth=0.25ex,
                 inner sep=0pt, nodes={inner xsep=0.3333em, inner ysep=0.3333em}]
    {
       \Loc(\alpha) & \Loc(\beta\circ\alpha) & \Loc(\beta)\\
       (\calD\p,?,\calE\pp) & (\calD,?,\calE\pp) & (\calD,?,\calE\p)\\
    };
    \draw[shorten >=0.12cm,shorten <=0.12cm,->,transform canvas={yshift=-0.4cm}] ($(m-1-1.east) + (0,0.8mm)$) --
    node[above,scale=0.75]{$\bfL\id$} ($(m-1-2.west) + (0,0.8mm)$);
    \draw[shorten >=0.12cm,shorten <=0.12cm,->,transform canvas={yshift=-0.4cm}] ($(m-1-2.west) - (0,0.8mm)$) --
    node[below,scale=0.75]{$\bfR\id$} ($(m-1-1.east) - (0,0.8mm)$);
    \draw[shorten >=0.12cm,shorten <=0.12cm,->,transform canvas={yshift=-0.4cm}] ($(m-1-2.east) + (0,0.8mm)$) --
    node[above,scale=0.75]{$\bfL\id$} ($(m-1-3.west) + (0,0.8mm)$);
    \draw[shorten >=0.12cm,shorten <=0.12cm,->,transform canvas={yshift=-0.4cm},] ($(m-1-3.west) - (0,0.8mm)$) --
    node[below,scale=0.75]{$\bfR\id$} ($(m-1-2.east) - (0,0.8mm)$);
    \draw[double, double distance=0.8mm] (m-2-1) -- (m-1-1);
    \draw[double, double distance=0.8mm] (m-2-2) -- (m-1-2);
    \draw[double, double distance=0.8mm] (m-2-3) -- (m-1-3);
\end{tikzpicture}\end{equation}
\end{proposition}
\begin{proof}
For $\omega$ the common core of the given cotorsion pairs, we have $\Ho\Loc(\alpha)\cong\calC\p\cap\calE\pp/\omega$ and
$\Ho\Loc(\beta\circ\alpha)\cong \calC\cap\calE\pp / \omega,$ and
$\bfL\id: \Ho\Loc(\alpha)\to\Ho\Loc(\beta\circ\alpha)$ is the canonical functor
$\calC\p\cap\calE\pp/\omega\to\calC\cap\calE\pp/\omega$, hence fully faithful. Similarly, $\bfR\id:
\Ho\Loc(\beta)\to\Ho\Loc(\beta\circ\alpha)$ is given by the embedding of $\calD\cap\calE\p/\omega$ into
$\calD\cap\calE\pp/\omega$, hence fully faithful.

It remains to prove the exactness of \eqref{eq:localizationcomposition}. Up to isomorphism in
$\Ho\Loc(\beta\circ\alpha)$, $\ker[\Ho\Loc(\beta\circ\alpha)\xrightarrow{\bfL\id}\Ho\Loc(\beta)]$ consists of those
$D\in\calD$ admitting a short exact sequence $0\to E\to D\p\to D\to 0$ with $E\in\calE$, $D\p\in\calD\p$ (recall that
these characterize the weakly trivial objects in $\Loc(\beta)$). As $\calE$ is the class of
trivially fibrant objects of $\Loc(\beta\circ\alpha)$, such a sequence already implies $D\cong D\p$ in
$\Ho\Loc(\beta\circ\alpha)$. Hence, up to isomorphism in $\Ho\Loc(\beta\circ\alpha)$,
$\ker[\Ho\Loc(\beta\circ\alpha)\xrightarrow{\bfL\id}\Ho\Loc(\beta)]$ consists of the objects of $\calD\p$, and the same
is true for $\image[\Ho\Loc(\alpha)\xrightarrow{\bfL\id}\Ho\Loc(\beta\circ\alpha)]$ by definition of $\bfL\id$.
\end{proof}
\begin{example}\label{ex_coctrgeneral} Proposition \ref{prop_locconcomp} applies to the chain of localization contexts
\begin{align}\label{eq:cogeneral}
(\perp,\dw\calW\cap\calF)\to(\wt{\calC},\dg\wt{\calW\cap\calF})\to(\dg\wt{\calC},\wt{\calW\cap\calF}),\\\label{eq:ctrgeneral}
(\wt{\calC\cap\calW},\dg\wt{\calF})\to(\dg\wt{\calC\cap\calW},\wt{\calF})\to(\dw\calC\cap\calW,\perp)
\end{align}
in the upper right resp. lower left corner of Figure \ref{fig_bunch}. The model structures
\begin{align}\label{eq:coctrgeneral}
(\dg\wt{\calC},?,\dw\calW\cap\calF)\quad\text{and}\quad (\dw\calC\cap\calW,?,\dg\wt{\calF})
\end{align}
associated with the composed localization contexts are generalizations of both the injective/projective models for
$\bfD(\scA)$ and the contraderived/coderived model structures:
\begin{enumerate}
\item[--] Choosing $(\calC,\calW,\calF)$ as
$(\scA,\scA,\inj(\scA))$, we have $(\dg\wt{\calC},?,\dw\calW\cap\calF)=\calM\uco(\scA)$ and $(\dw\calC\cap\calW,?,\dg\wt{\calF})
= \calM\uinj(\scA)$.
\item[--] Choosing $(\calC,\calW,\calF)$ as $(\proj(\scA),\scA,\scA)$ in case $\scA$ has enough
projectives, we have $(\dg\wt{\calC},?,\dw\calW\cap\calF)=\calM\uproj(\scA)$ and $(\dw\calC\cap\calW,\dg\wt{\calF}) =
\calM\uctr(\scA)$.
\end{enumerate}
Further, in the case of $(\scA,\scA,\inj(\scA))$ the localization sequences induced by
\eqref{eq:cogeneral} is the known one $\bfD(\scA)\rightleftarrows \bfK(\inj(\scA))\rightleftarrows\bfK\lac(\inj(\scA))$,
while the one induced by \eqref{eq:ctrgeneral} is the trivial localization sequence
$0\rightleftarrows\bfD(\scA)\rightleftarrows\bfD(\scA)$. Similarly, in the case of $(\proj(\scA),\scA,\scA)$, the
localization sequence associated with \eqref{eq:cogeneral} is trivial, while the one associated with
\eqref{eq:ctrgeneral} is the classical one
$\bfK\lac(\proj(\scA))\rightleftarrows\bfK(\proj(\scA))\rightleftarrows\bfD(\scA)$.
\end{example}
\begin{fact}\label{fact_locconsquare}
Suppose given a square of localization contexts
\begin{equation*}\begin{tikzpicture}[description/.style={fill=white,inner sep=2pt}]
    \matrix (m) [matrix of math nodes, row sep=2.5em,
                 column sep=3.5em, text height=1.5ex, text depth=0.25ex,
                 inner sep=0pt, nodes={inner xsep=0.3333em, inner ysep=0.3333em}]
    {
      (\calD\ppp,\calE\ppp) & (\calD\pp,\calE\pp) \\
       (\calD\p,\calE\p) & (\calD,\calE)\\
    };
    \draw[->] (m-1-1) -- node[above,scale=0.75]{$\alpha$} (m-1-2);
    \draw[->] (m-2-1) -- node[below,scale=0.75]{$\gamma$} (m-2-2);
    \draw[->] (m-1-1) -- node[left,scale=0.75]{$\beta$} (m-2-1);
    \draw[->] (m-1-2) -- node[right,scale=0.75]{$\delta$} (m-2-2);
    \draw[->] (m-1-1) -- node[description,scale=0.75]{$\veps$} (m-2-2);
\end{tikzpicture}\end{equation*}
Then their localizations fit into a diagram of identity Quillen adjunctions
\begin{align}\label{eq:pseudobutterfly}
\begin{tikzpicture}[baseline,description/.style={fill=white,inner sep=2pt}]
    \matrix (m) [matrix of math nodes, row sep=2em,
                 column sep=4em, text height=1.5ex, text depth=0.25ex,
                 inner sep=0pt, nodes={inner xsep=0.3333em, inner ysep=0.3333em}]
    {
       (\calD\p,?,\calE\ppp) \pgfmatrixnextcell\pgfmatrixnextcell (\calD,?,\calE\p)\\
       \pgfmatrixnextcell (\calD,?,\calE\ppp) \pgfmatrixnextcell \\
       (\calD,?,\calE\pp) \pgfmatrixnextcell\pgfmatrixnextcell (\calD\pp,?,\calE\ppp)\\
    };
    \draw[->,shorten <=0.1cm,shorten >=0.1cm] ($(m-1-1.south) + (1mm,0)$) -- node[scale=0.75,right]{$\adjL$} ($(m-3-1.north) + (1mm,0)$);
    \draw[->,shorten <=0.1cm,shorten >=0.1cm] ($(m-3-1.north) - (+1mm,0)$) -- node[scale=0.75,left]{$\adjR$} ($(m-1-1.south) - (1mm,0)$);
    \draw[->,shorten <=0.1cm,shorten >=0.1cm] ($(m-1-3.south) + (1mm,0)$) -- node[scale=0.75,right]{$\adjR$} ($(m-3-3.north) + (1mm,0)$);
    \draw[->,shorten <=0.1cm,shorten >=0.1cm] ($(m-3-3.north) + (-1mm,0)$) -- node[scale=0.75,left]{$\adjL$} ($(m-1-3.south) - (1mm,0)$);
    \draw[<-,shorten <=0.2cm,transform canvas={yshift=-0.5mm}] ($(m-1-1.south east) - (4mm,0)$) --
    node[scale=0.75,below]{$\adjR$} ($(m-2-2.north west) + (0mm,0)$);
    \draw[<-,shorten <=0.2cm] ($(m-2-2.north west) + (4mm,0)$) -- node[scale=0.75,above]{$\adjL$} ($(m-1-1.south east) - (0mm,0)$);
    \draw[<-,shorten >=0.2cm] ($(m-3-1.north east) + (0mm,0)$) -- node[scale=0.75,below]{$\adjL$} ($(m-2-2.south west) + (4mm,0)$);
    \draw[<-,shorten >=0.2cm,transform canvas={yshift=+0.5mm}] ($(m-2-2.south west) + (0mm,0)$) -- node[scale=0.75,above]{$\adjR$} ($(m-3-1.north east) - (4mm,0)$);
    \draw[<-,shorten >=0.2cm,transform canvas={yshift=-0.5mm}] ($(m-2-2.north east) + (0mm,0)$) -- node[scale=0.75,below]{$\adjR$} ($(m-1-3.south west) + (4mm,0)$);
    \draw[<-,shorten >=0.2cm] ($(m-1-3.south west) + (0mm,0)$) -- node[scale=0.75,above]{$\adjL$} ($(m-2-2.north east) - (4mm,0)$);
    \draw[<-,shorten <=0.2cm] ($(m-2-2.south east) - (4mm,0)$) -- node[scale=0.75,below]{$\adjL$} ($(m-3-3.north west) + (0mm,0)$);
    \draw[<-,shorten <=0.2cm,transform canvas={yshift=+0.5mm}] ($(m-3-3.north west) + (4mm,0)$) -- node[scale=0.75,above]{$\adjR$} ($(m-2-2.south east) + (0mm,0)$);
\end{tikzpicture}
\end{align}
\end{fact}

\begin{proof}[Proof of Theorem \ref{thm_chainbutterfly}]
According to Fact \ref{fact_locconsquare}, each oriented square in Figure \ref{fig_bunch} gives rise to a diagram of the form
\eqref{eq:pseudobutterfly} in which the upper and lower rows are exact by Proposition \ref{prop_locconcomp}. Applying this to the two tilted squares in Figure
\eqref{fig_bunch} gives rise to the diagram of model structures depicted in Figure \ref{fig_chainbutterfly}.

The left vertical Quillen adjunctions are Quillen equivalences by Theorem \ref{thm_lifttocomplexes}; in fact, they share
their classes of bifibrant objects, as do the two model structures connected by the middle vertical adjunction, which is therefore a
Quillen equivalence, too. Moreover, this argument also proves one half of the wing-commutativity condition
\ref{item:butterflywingsgen} from Definition \ref{def_butterfly} of butterflies. Similarly, the right vertical
adjunctions are Quillen equivalences by Proposition \ref{prop_bunchderived}; in fact, they share their classes of weakly
trivial objects, proving the remaining part of condition \ref{item:butterflywingsgen}.

The remaining properties are established by Proposition \ref{prop_locconcomp}, so the proof is finished.
\end{proof}

\section{Proof of Theorem \ref{thm_leftrightstabilization}}

Denote $\calT$ the common homotopy category of the two middle model structures in Figure \ref{fig_chainbutterfly}; explicitly, this is
the homotopy category of complexes with components in $\omega = \calC\cap\calW\cap\calF$ which belong to both
$\dg\wt{\calC}$ and $\dg\wt{\calF}$. Passing to homotopy categories in Figure \ref{fig_chainbutterfly} now yields a
recollement $$\Ho(\calM)\bigrec\calT\bigrec\bfD(\scA)$$ and we claim that the induced stabilization functor $\comp:
\bfD(\scA)\to\Ho(\calM)$ makes the following diagram commutative:
\begin{equation*}
\begin{tikzpicture}[baseline,description/.style={fill=white,inner sep=2pt}]
    \matrix (m) [matrix of math nodes, row sep=1.5em,
                 column sep=2.5em, text height=1.5ex, text depth=0.25ex,
                 inner sep=0pt, nodes={inner xsep=0.3333em, inner ysep=0.3333em}]
    {
       \Ho(\calM) && \bfD(\scA)\\
       & \scA & \\
    };
    \draw[->,rounded corners=4mm] (m-2-2) -| (m-1-3);
    \draw[->,rounded corners=4mm] (m-2-2) -| (m-1-1);
    \draw[->] (m-1-3) -- node[above,scale=0.75]{$\comp$} (m-1-1);
\end{tikzpicture}
\end{equation*}
We begin with some generalities on stabilization functors associated to recollements:
\begin{definition}[\protect{see \cite[\S 5]{Krause_StableDerived}}]\label{def_stabilizationfunctor}
Given a recollement $\calT\p\rec\calT\rec\calT\pp$ of triangulated categories, the functors $I_\rho Q_\lambda:
\calT\pp\to\calT\p$ resp. $I_\lambda Q_\rho: \calT\pp\to\calT\p$ are called the \emph{left resp. right stabilization
  functors} associated to the recollement.
\end{definition}
\begin{fact}\label{fact_stabiso}
For any $X\pp\in\calT\pp$ there is a non-canonical isomorphism
$$I_\lambda Q_\rho X\pp\ \cong\ \Sigma I_\rho Q_\lambda X\pp.$$
\end{fact}
\begin{proof}
For $X\in\calT$ the localization sequence $\calT\p\rightleftarrows\calT\rightleftarrows\calT\pp$ induces a
non-canonical distinguished triangle $Q_\lambda Q X\to X\to I I_\lambda X\to \Sigma Q_\lambda Q X$, which for
$X= Q_\rho X\pp$ with $X\pp\in\calT\pp$ transforms into $Q_\lambda X\pp\to Q_\rho X\pp\to
I I_\lambda Q_\rho X\pp\to \Sigma Q_\lambda X\pp$. Applying $I_\rho$ from the left annihilates $Q_\rho X$ and
hence yields an isomorphism $I_\lambda Q_\rho X\pp\cong\Sigma I_\rho Q_\lambda X\pp$.
\end{proof}
Generalizing the inclusions $\Acyc^-\subset\calW\uctr\subset\Acyc\supset\calW\uco\supset\Acyc^+$, we have:
\begin{fact}\label{fact_coctrgeneral}
The model structures \eqref{eq:coctrgeneral} have the following properties:
\begin{enumerate}
\item The class $\calW$ of weakly trivial objects in the model structure $(\dg\wt{\calC},?,\dw\calW\cap\calF)$ satisfies $\Acyc^+(\scA)\subseteq\calW\subseteq\Acyc(\scA)$.
\item The class $\calW$ of weakly trivial objects in the model structure $(\dw\calC\cap\calW,?,\dg\wt{\calF})$
  satisfies $\Acyc^-(\scA)\subseteq\calW\subseteq\Acyc(\scA)$.
\end{enumerate}
\end{fact}
\begin{proof}
Recall that the model structure $(\dg\wt{\calC},?,\dw\calW\cap\calF)$ arises as the localization of the composed
localization context in \eqref{eq:cogeneral}. Therefore, by Gillespie's Theorem \ref{thm_gillespie}, its class $\calW$ of weakly
trivial complexes consists of those $X\in\Ch(\scA)$ which admit a short exact sequence $0\to X\to F\to C\to 0$ with
$F\in\wt{\calW\cap\calF}$ and $C\in {^{\perp}}[\dw\calW\cap\calF]$. Now $\wt{\calW\cap\calF}\subseteq\Acyc(\scA)$
 by definition, and ${^{\perp}}[\dw\calW\cap\calF]\subseteq\Acyc(\scA)$ as witnessed by the upper right dashdotted arrow
in Figure \ref{fig_bunch}, so $\calW\subseteq\Acyc(\scA)$ by the $2$-out-of-$3$ property of $\Acyc(\scA)$. Conversely, suppose
$X\in\Acyc^+(\scA)$. Applying \cite[Resolution Lemma 1.4.4]{Becker_ModelsForSingularityCategories} to the cotorsion pair
$(\calC,\calW\cap\calF)$ and all short exact sequences $0\to \syz^n X\to X^n\to \syz^{n+1} X\to 0$ we can construct a short
exact sequence $0\to X\to F\to C\to 0$ in $\Ch(\scA)$ with $F\in\wt{\calW\cap\calF}\cap\Ch^+(\scA)$ and
$C\in\wt{\calC}\cap\Ch^+(\scA)$. Since $\wt{\calC}\cap\Ch^+(\scA)\subseteq {^{\perp}}[\dw\calW\cap\calF]$, it follows
that $X\in\calW$. This finishes the proof of statement (i), and (ii) is analogous.
\end{proof}

\begin{proof}[Proof of Theorem \ref{thm_leftrightstabilization}]
\textit{Step 1:} Firstly, we note that the derived functor
\begin{align}\label{eq:leftadjoint}
Q_\lambda:
\Ho(\dg\wt{\calC\cap\calW},\Acyc(\scA),\dg\wt{\calF})\to\Ho(\dg\wt{\calC}\cap\dw\calC\cap\calW,?,\dg\wt{\calF})
\end{align} may be computed naively on
$\Ch^-(\scA)$: Namely, it can be computed through any resolution by a quasi-isomorphic bounded above complex with entries in
$\calC\cap\calW$, and by Fact \ref{fact_coctrgeneral} such a resolution is still a weak equivalence in
$(\dw\calC\cap\calW,?,\dg\wt{\calF})$, hence a fortiori also in $(\dg\wt{\calC}\cap\dw\calC\cap\calW,?,\dg\wt{\calF})$.

\textit{Step 2:} We claim that the stabilization functor in question annihilates all $X\in\calW$. For that, step 1
and the exactness (Proposition \ref{prop_locconcomp}) of the sequence of functors
\begin{equation*}\begin{tikzpicture}[description/.style={fill=white,inner sep=2pt}]
    \matrix (m) [matrix of math nodes, row sep=1.5em,
                 column sep=2.5em, text height=1.5ex, text depth=0.25ex,
                 inner sep=0pt, nodes={inner xsep=0.3333em, inner ysep=0.3333em}]
    {
\Ho(\dg\wt{\calC}\cap\dw\calC\cap\calW,\Acyc(\scA),[\wt{\calC}\cap\dw\calC\cap\calW]^{\perp})\\
\Ho(\dg\wt{\calC}\cap\dw\calC\cap\calW,?,\dg\wt{\calF})\\
\Ho(\wt{\calC}\cap\dw\calC\cap\calW,?,\wt{\calF})       \\
    };
    \draw[->,shorten <=0.5mm, shorten >=0.5mm] (m-1-1) -- node[scale=0.75,right]{$\bfR\id$}  (m-2-1);
    \draw[->,shorten <=0.5mm, shorten >=0.5mm] (m-2-1) -- node[scale=0.75,right]{$\bfR\id$}  (m-3-1);
\end{tikzpicture}\end{equation*}
show that it suffices to prove that any $X\in\calW\subset\Ch(\scA)$ belongs to
$[\wt{\calC}\cap\dw\calC\cap\calW]^{\perp}$ up to weak equivalence in
$(\dg\wt{\calC}\cap\dw\calC\cap\calW,?,\dg\wt{\calF})$. Now, the presence of enough injectives with respect to
$(\calC\cap\calW,\calF)$ shows that $X$ admits a resolution $\iota: X\rightarrowtail F$ with $F\in\Ch^{\geq 0}(\calF)$ such that
$\syz^k F\in \calC\cap\calW$ for $k>0$. The thickness of $\calW$ then implies that even $F\in\Ch^{\geq
  0}(\calW\cap\calF)\subset\dg\wt{\calW\cap\calF}$, and moreover $Z := \coker(\iota)\in\wt{\calC\cap\calW}$. Since
$\wt{\calC\cap\calW}$ are the trivially cofibrant objects in
$(\dg\wt{\calC}\cap\dw\calC\cap\calW,?,\dg\wt{\calF})$ and
$\dg\wt{\calW\cap\calF}\subseteq[\wt{\calC}\cap\dw\calC\cap\calW]^{\perp}$ as witnessed by the right vertical arrows in
Figure \ref{fig_chainbutterfly}, it follows that $X\cong F\in [\wt{\calC}\cap\dw\calC\cap\calW]^{\perp}$ in
$\Ho(\dg\wt{\calC}\cap\dw\calC\cap\calW,?,\dg\wt{\calF})$ as claimed.

\textit{Step 3:} Since any $X\in\scA$ admits a functorial resolution of the form $0\to X\to F\to C\to 0$ with $F\in\calF$ and
$C\in\calC\cap\calW$, the second step shows that it suffices to prove the commutativity of
\eqref{eq:realizationsolution} when restricted to the fibrant objects $\calF\subseteq\scA$. In this case,
by definition as well as step 1, both the left stabilization $$\Ho(\dg\wt{\calC\cap\calW},\Acyc(\scA),\dg\wt{\calF})\longrightarrow\Ho(\wt{\calC}\cap\dw\calC\cap\calW,?,\dg\wt{\calF})$$
and the functor
$$\bfR\iota^0: \Ho(\calM)\to \Ho(\wt{\calC}\cap\dw\calC\cap\calW,?,\dg\wt{\calF})$$
can be computed naively, and the commutativity of \eqref{eq:realizationsolution} follows.
\end{proof}

Dually, the composition of the right stabilization functor associated to the upper butterfly in Figure
\ref{fig_chainbutterfly} and the equivalence $\bfR \syz^0$ from \ref{eq:frobeniuslift2} in Theorem \ref{thm_lifttocomplexes} gives
another functor $\bfD(\scA)\to\Ho(\calM)$ making \eqref{eq:realizationsolution} commutative. Comparing the two functors
$\bfD(\scA)\to\Ho(\calM)$ obtained this way, Fact \ref{fact_stabiso} and $\bfL
\cosyz^0\cong\Sigma\bfR \syz^0$ show that they are pointwise non-canonically isomorphic; to show that they are even naturally
isomorphic we'd need to find an enhancement of Fact \ref{fact_stabiso}.

\bibliography{bibliography}{}
\bibliographystyle{halpha}

\end{document}